\documentclass[11pt,a4paper,reqno]{amsart}

\usepackage{geometry}
\usepackage{ stmaryrd }

\usepackage{hyperref}

\usepackage[T1]{fontenc}
\usepackage[latin1]{inputenc}

\usepackage{setspace}
\usepackage{indentfirst}

\usepackage{amsmath,amssymb}

\usepackage{stmaryrd}

\usepackage{graphicx}
\usepackage[all]{xy}
\def\dar[#1]{\ar@<2pt>[#1]\ar@<-2pt>[#1]}

\usepackage{amsthm}

\theoremstyle{plain}
\newtheorem{prop}{Proposition}[section]
\newtheorem{lem}[prop]{Lemma}

\newtheorem{cor}[prop]{Corollary}
\newtheorem{thm}[prop]{Theorem}

\newtheorem*{prop*}{Proposition}
\newtheorem*{lem*}{Lemma}
\newtheorem*{sublem*}{Sublemma}
\newtheorem*{cor*}{Corollary}
\newtheorem*{thm*}{Theorem}
\newtheorem*{hypo*}{Hypothesis}
\newtheorem*{question*}{Question}
\newtheorem*{conjecture*}{Conjecture}
\newtheorem*{scholum*}{Scholum}
\newtheorem{defn}[prop]{Definition}
\newtheorem*{defn*}{Definition}

\theoremstyle{definition}

\newtheorem*{con*}{Construction}
\newtheorem*{note*}{Note}

\theoremstyle{remark}

\newtheorem*{warning*}{Warning}
\newtheorem*{shortnote*}{Note}
\newtheorem*{claim*}{Claim}
\newtheorem*{axiom*}{Axiom}

\newtheoremstyle{slanted}
  {3pt}
  {3pt}
  {\slshape}
  {}
  {\bfseries}
  {.}
  {.5em}
  {}

\theoremstyle{slanted}
\newtheorem{example}[prop]{Example}
\newtheorem{examples}[prop]{Examples}
\newtheorem*{example*}{Example}
\newtheorem*{examples*}{Examples}

\newtheorem*{ex*}{Example}
\newtheorem*{exs*}{Examples}
\newtheorem{remark}[prop]{Remark}

\newtheorem*{remark*}{Remark}
\newtheorem*{remarks*}{Remarks}

\newtheorem*{rmk*}{Remark}
\newtheorem*{rmks*}{Remarks}

\DeclareMathOperator{\pr}{pr}

\DeclareMathOperator{\RE}{Re} 
\DeclareMathOperator{\IM}{Im} 

\DeclareMathOperator{\Id}{Id}

\newcommand{\beq}[1]{\begin{equation}\label{#1}}
\newcommand{\eeq}{\end{equation}}

\newcommand{\CC}{\mathbb{C}}

\newcommand{\lie}[2]{[#1,#2]} 

\newcommand{\oplie}{\mathfrak{A}}
\newcommand{\Glanon}{Glanon }

\newcommand{\rond}{\circ}
\newcommand{\smalcirc}{\mbox{\tiny{$\circ$}}}
\newcommand{\cc}[1]{\overline{#1}} 

\newcommand{\sections}[1]{\Gamma(#1)}

\newcommand{\imaginary}{\mathbf{i}}

\newcommand{\bemol}{^{\flat}}
\newcommand{\inv}{^{-1}}

\newcommand{\TC}{T_{\mathbb{C}}}

\newcommand{\gm}{\Gamma}

\newcommand{\hf}[1]{\mathcal{O}_{#1}} 
\newcommand{\shs}[1]{\mathcal{#1}} 
\newcommand{\com}{\bullet} 
\newcommand{\be }{\begin{eqnarray*}}
\newcommand{\ee }{\end{eqnarray*}}

\newcommand{\lon }{\longrightarrow }

\newcommand{\C}{{\mathbb{C}}}

\newcommand{\J}{{\mathcal{J}}}
\newcommand{\tg}{{\mathsf{t}}}
\newcommand{\s}{{\mathsf{s}}}
\newcommand{\m}{{\mathsf{m}}}
\newcommand{\rr}{{\rightrightarrows}}
\newcommand{\T}{{\mathbb{T}}}
\newcommand{\I}{{\mathcal{I}}}
\newcommand{\R}{\mathbb{R}}

\newcommand{\mx}{\mathfrak{X}}
\newcommand{\li}[1]{\mathfrak{#1}}
\newcommand{\dr}{\mathbf{d}}
\newcommand{\ldr}[1]{{{\pounds}}_{#1}}
\newcommand{\ip}[1]{{\mathbf{i}}_{#1}}
\newcommand{\an}[1]{\arrowvert_{#1}}

\allowdisplaybreaks[1]

\begin{document}

\title{Glanon groupoids}
\thanks{Partially supported by a Dorothea-Schl\"ozer fellowship of the
  University of G\"ottingen, Swiss NSF grant 200021-121512, NSF grants
  DMS-0801129 and DMS-1101827, NSA grant H98230-12-1-0234.}

\author{Madeleine Jotz}
\address{Mathematisches Institut, Georg-August Universit\"at G\"ottingen} 
\email{\href{mailto:mjotz@uni-math.gwdg.de}{\texttt{mjotz@uni-math.gwdg.de}}}

\author{Mathieu Sti\'enon}
\address{Department of Mathematics, Penn State University} 
\email{\href{mailto:stienon@math.psu.edu}{\texttt{stienon@math.psu.edu}}}

\author{Ping Xu}
\address{Department of Mathematics, Penn State University} 
\email{\href{mailto:ping@math.psu.edu}{\texttt{ping@math.psu.edu}}}
\maketitle

\begin{abstract} 
  We introduce the notion of Glanon groupoids, which are Lie groupoids
  equipped with multiplicative generalized complex structures.  It
  combines symplectic groupoids, holomorphic Lie groupoids and
  holomorphic Poisson groupoids into a unified framework.  Their
  infinitesimal, Glanon Lie algebroids are studied.  We prove that
  there is a bijection between Glanon Lie algebroids and source-simply
  connected and source-connected Glanon groupoids.  As a consequence,
  we recover various integration theorems and obtain the integration
  theorem for holomorphic Poisson groupoids.
 \end{abstract}


\section{Introduction}
In their study of quantization theory, Karasev \cite{Karasev89},
Weinstein \cite{Weinstein87}, and Zakrzewski \cite{Zakrzewski90a,
  Zakrzewski90b} independently introduced the notion of symplectic
groupoids.  By a symplectic groupoid, we mean a Lie groupoid equipped
with a multiplicative symplectic two-form on the space of morphisms.
It is a classical theorem that the unit space of a symplectic groupoid
is naturally a Poisson manifold \cite{CoDaWe87}.  Indeed the Lie
algebroid of a symplectic groupoid $\Gamma \rr P$ is naturally
isomorphic to $(T^*P)_\pi$, the canonical Lie algebroid associated to
the Poisson manifold $(P, \pi)$. In \cite{MaXu00}, Mackenzie and one
of the authors proved that, for a given Poisson manifold $(P, \pi)$,
if the Lie algebroid $(T^*P)_\pi$ integrates to a $\s$-connected and
$\s$-simply connected Lie groupoid $\Gamma \rr P$, then $\Gamma$ is
naturally a symplectic groupoid. The symplectic groupoid structure on
$\Gamma$ was also obtained by Cattaneo-Felder \cite{CaFe01} using the
Poisson sigma model.  The general integration condition for symplectic
groupoids was completely solved recently by Crainic-Fernandes
\cite{CrFe04}.

Recently, there has been increasing interest in holomorphic Lie
algebroids and holomorphic Lie groupoids.  It is a very natural
question to find the integrability condition for holomorphic Lie
algebroids.  In \cite{LaStXu08}, together with Laurent-Gengoux, two of
the authors studied holomorphic Lie algebroids and their relation with
real Lie algebroids. In particular, they proved that associated to any
holomorphic Lie algebroid $A$, there is a canonical real Lie algebroid
$A_R$ such that the inclusion $\shs{A}\to \shs{A}_\infty$ is a
morphism of sheaves.  Here $\shs{A}$ and $\shs{A}_\infty$ denote,
respectively, the sheaf of holomorphic sections and the sheaf of
smooth sections of $A$.  In other words, a holomorphic Lie algebroid
can be considered as a holomorphic vector bundle $A\to X$ whose
underlying real vector bundle is endowed with a Lie algebroid
structure such that, for any open subset $U\subset X$,
$\lie{\shs{A}(U)}{\shs{A}(U)}\subset\shs{A}(U)$ and the restriction of
the Lie bracket $\lie{\cdot}{\cdot}$ to $\shs{A}(U)$ is $\CC$-linear.
In particular, they proved that a holomorphic Lie algebroid $A$ is
integrable as a holomorphic Lie algebroid if and only if its
underlying real Lie algebroid $A_R$ is integrable as a real Lie
algebroid \cite{LaStXu09}.

In this paper, we introduce the notion of \Glanon groupoids, as a
unification of symplectic groupoids and holomorphic groupoids. By a
\Glanon groupoid, we mean a Lie groupoid equipped with a generalized
complex structure in the sense of Hitchin \cite{Hitchin03} which is
multiplicative.  Recall that a generalized complex structure on a
manifold $M$ is a smooth bundle map $\J: TM\oplus T^*M\to TM\oplus
T^*M$ such that $\J^2=-\Id$, $\J\J^*=\Id$, and the $+i$-eigenbundle of
$\J$ is involutive with respect to the Courant bracket, or the
Nijenhuis tensor of $\J$ vanishes. A generalized complex structure
$\J$ on a Lie groupoid $\Gamma\rr M$ is said to be multiplicative if
$\J$ is a Lie groupoid automorphism with respect to the Courant
groupoid $T\Gamma\oplus T^*\Gamma\rr TM\oplus A^*$ in the sense of
Mehta \cite{Mehta09}. When $\J$ is the generalized complex structure
associated to a symplectic structure on $\Gamma$, it is clear that
$\J$ is multiplicative if and only if the symplectic two-form is
multiplicative. Thus its \Glanon groupoid is equivalent to a
symplectic groupoid. On the other hand, when $\J$ is the generalized
complex structure associated to a complex structure on $\Gamma$, $\J$
is multiplicative if and only if the complex structure on $\Gamma$ is
multiplicative. Thus one recovers a holomorphic Lie groupoid.  On the
infinitesimal level, a \Glanon groupoid corresponds to a \Glanon Lie
algebroid, which is a Lie algebroid $A$ equipped with a compatible
generalized complex structure, i.e. a generalized complex structure
$\J_A: TA\oplus T^*A \to TA\oplus T^*A$ which is a Lie algebroid
automorphism with respect to the Lie algebroid $TA\oplus T^*A\to
TM\oplus A^*$. More precisely, we prove the following main result:

\vskip.2in

\noindent{\bf Theorem A.} {\it If $\Gamma$ is a $\s$-connected and
  $\s$-simply connected Lie groupoid with Lie algebroid $A$, then
  there is a one-one correspondence between \Glanon groupoid
  structures on $\Gamma$ and \Glanon Lie algebroid structures on $A$.}
\vskip.2in

As a consequence, we recover the following standard results
\cite{MaXu00, LaStXu09}: 

\noindent{\bf Theorem B.} {\it Let $(P, \pi )$ be a Poisson manifold.
  If $\Gamma$ is a $\s$-connected and $\s$-simply connected Lie
  groupoid integrating the Lie algebroid $(T^*P)_\pi$, then $\Gamma$
  automatically admits a symplectic groupoid structure.}
\vskip.2in

\noindent{\bf Theorem C.} {\it If $\Gamma$ is a $\s$-connected and
  $\s$-simply connected Lie groupoid integrating the underlying real
  Lie algebroid $A_R$ of a holomorphic Lie algebroid $A$, then
  $\Gamma$ is a holomorphic Lie groupoid.}
\vskip.2in

Another important class of \Glanon groupoids is when the generalized
complex structure $\J$ on $\Gamma$ has the form
\[\mathcal{J}= \begin{pmatrix}
                N & \pi^{\sharp} \\
                0 & - N^*  
              \end{pmatrix}.\] It is simple to see that in this case
              $(\Gamma\rr M, \J)$ being a \Glanon groupoid is
              equivalent to $\Gamma$ being a holomorphic Poisson
              groupoid. On the other hand, on the infinitesimal level,
              one proves that their corresponding \Glanon Lie
              algebroids are equivalent to holomorphic Lie
              bialgebroids. Thus as a consequence, we prove the
              following

              \noindent{\bf Theorem D.} {\it Given a holomorphic Lie
                bialgebroid $(A, A^*)$, if the underlying real Lie
                algebroid $A_R$ integrates to a $\s$-connected and
                $\s$-simply connected Lie groupoid $\Gamma$, then
                $\Gamma$ is a holomorphic Poisson groupoid.}
\vskip.2in

A special case of this theorem was proved in \cite{LaStXu09} using a
different method, namely, when the holomorphic Lie bialgebroid is $(
(T^* X)_\pi , TX)$, the one corresponding to a holomorphic Poisson
manifold $(X, \pi)$.  In this case, it was proved that when the
underlying real Lie algebroid of $(T^* X)_\pi$ integrates to a
$\s$-connected and $\s$-simply connected Lie groupoid $\Gamma$, then
$\Gamma$ is automatically a holomorphic symplectic groupoid.  However,
according to our knowledge, the integrating problem for general
holomorphic Lie bialgebroids remain open.  Indeed, solving this
integration problem is one of the main motivation behind our study of
\Glanon groupoids.

It is known that the base manifold of a generalized complex manifold
is automatically a Poisson manifold \cite{BaSt08, Gualtieri03}. Thus
it follows that a \Glanon groupoid is automatically a (real) Poisson
groupoid.  On the other hand, a \Glanon Lie algebroid $A$ admits a
Poisson structure, which can be shown a linear Poisson structure.
Therefore $A^*$ is a Lie algebroid.  We prove that $(A, A^*)$ is
indeed a Lie bialgebroid, which is the infinitesimal of the associated
Poisson groupoid of the \Glanon groupoid.

Note that in this paper we confine ourselves to the standard Courant
groupoid $T\Gamma \oplus T^*\Gamma$. We can also consider the twisted
Courant groupoid $(T\Gamma \oplus T^*\Gamma)_H$, where $H$ is a
multiplicative closed $3$-form. This will be discussed somewhere else.
Also note that a multiplicative generalized complex structure on the
groupoid $\Gamma$ induces a pair of maps $j_u:TM\oplus A^*\to TM\oplus
A^*$ and $j_A:A\oplus T^*M\to A\oplus T^*M$ on the units and on the
core, respectively. Since a multiplicative generalized complex structure 
is the same as two complex conjugated multiplicative Dirac structures 
on the groupoid, we get following the results in \cite{Jotz12c}, \cite{DrJoOr12}
that a multiplicative generalized complex structure 
is the same as two complex conjugated  Manin pairs over the set of units $M$
(in the sense of \cite{BuIgSe09}).
The detailed study of properties of the maps $j_A,j_u$, the Manin pairs, and the associated Dorfman connections
will be investigated in the spirit of \cite{LaStXu08} in a future project.

\medskip

{\bf Acknowledgments.}  We would like to thank Camille
Laurent-Gengoux, Rajan Mehta, and Cristian Ortiz for useful
discussions.  We also would like to thank several institutions for
their hospitality while work on this project was being done: Penn
State University (Jotz), and IHES and Beijing International Center for
Mathematical Research (Xu).  Special thanks go to the Michea family
and our many friends in Glanon, whose warmth provided us constant
inspiration for our work in the past.  Many of our friends from Glanon
have left mathematics, but our memories of them serve as a reminder
that a mathematician's primary role should be to elucidate and enjoy
the beauty of mathematics.  We dedicate this paper to them.

\subsection{Notation}
In the following, $\Gamma\rr M$ will always be a Lie groupoid with set
of arrows $\Gamma$, set of objects $M$, source and target
$\s,\tg:\Gamma\to M$ and object inclusion map $\epsilon:M\to \Gamma$.
The product of $g, h\in \Gamma$ with $\s(g)=\tg(h)$ will be written
$\m(g,h)=g\star h$ of simply $gh$.

The Lie functor that sends a Lie groupoid to its Lie algebroid and Lie
groupoid morphisms to Lie algebroid morphisms is $\mathsf A$. For
simplicity, we will write $\mathsf A(\Gamma)=A$.  The Lie algebroid
$q_A:A\to M$ is identified with $T^\s_M\Gamma$, the bracket
$[\cdot\,,\cdot]_A$ is defined with the right invariant vector fields
and the anchor $\rho_A=\rho$ is the restriction of $T\tg$ to
$A$. Hence, as a manifold, $A$ is embedded in $T\Gamma$. The inclusion
is $\iota:A\to T\Gamma$.  Note that the Lie algebroid of $\Gamma\rr
M$ can alternatively be defined using left-invariant sections. Given
$a\in\Gamma(A)$, the right-invariant section corresponding to $a$ will
simply be written $a^r$, i.e., $a^r(g)=TR_ga(\tg(g))$ for all $g \in
\Gamma$.  We will write $a^l$ for the left-invariant vector field
defined by $a$, i.e., $a^l(g)=-T(L_g\circ\mathsf i)(a(\s(g)))$ for all
$g\in \Gamma$.

The projection map of a vector bundle $A\to M$ will always be written
$q_A:A\to M$, unless specified otherwise.  For a smooth manifold $M$,
we fix once and for all the notation $p_M:=q_{TM}:TM\to M$ and
$c_M:=q_{T^*M}:T^*M\to M$.  We will write $\mathsf P_M$ for the direct
sum $TM\oplus T^*M$ of vector bundles over $M$, and $\pr_M$ for the
canonical projection $\mathsf P_M\to M$.

A bundle morphism $\mathsf P_M\to\mathsf P_M$, for a manifold $M$, 
will always be meant to be over the identity on $M$.

\section{Preliminaries}
\subsection{Dirac structures}
Let $A\to M$ be a vector bundle with dual bundle $A^*\to M$.  The
natural pairing $A\oplus A^*\to \R$, $(a_m,\xi_m)\mapsto \xi_m(a_m)$
will be written $\lbag\cdot\,,\cdot\rbag_A$ or
$\lbag\cdot\,,\cdot\rbag_{q_A}$
if the vector bundle structure needs to be specified.  The direct sum $A\oplus
A^*$ is endowed with a canonical fiberwise pairing $(\cdot\,,\cdot)_A$
given by
 \begin{equation}
   ((a_m,\xi_m), (b_m,\eta_m))_A=\lbag b_m, \xi_m\!\rbag_A+\lbag\! a_m, \eta_m\rbag_A
\label{sym_bracket}
\end{equation}
for all $m\in M$, $a_m,b_m\in A_m$ and $\xi_m,\eta_m\in A^*_m$.
   
In particular, the \emph{Pontryagin bundle} $\mathsf{P}_M:=TM\oplus
T^* M$ of a smooth manifold $M$ is endowed with the pairing
$(\cdot\,,\cdot)_{TM}$, which will be written as usual $\langle
\cdot\,,\cdot\rangle_M$.

The orthogonal space relative to the pairing $(\cdot\,,\cdot)_A$ of a
subbundle $\mathsf E\subseteq A\oplus A^*$ will be written $\mathsf
E^\perp$ in the following.  An \emph{almost Dirac structure} \cite{Courant90a} on $M $ is a Lagrangian vector subbundle
$\mathsf{D} \subset \mathsf{P}_M $. That is, $ \mathsf{D}$ coincides
with its orthogonal relative to \eqref{sym_bracket}, $\mathsf
D=\mathsf D^\perp$, so its fibers are necessarily $\dim M
$-dimensional.

The set of sections $\Gamma(\mathsf P_M)$ of the Pontryagin bundle of $M$
is endowed with the Courant bracket, given by
\begin{align}
  \llbracket X+\alpha, Y+\beta \rrbracket &=[X,Y]+\left(
    \ldr{X}\beta-\ldr{Y}\alpha+\frac{1}{2}
    \dr(\alpha(Y)-\beta(X))\right)\label{wrong_bracket}\\
  &= [X,Y]+ \left(\ldr{X}\beta-\ip{Y}\dr\alpha-\frac{1}{2} \dr\langle
    (X,\alpha), (Y,\beta)\rangle_M\right)\nonumber
\end{align}
for all $X+\alpha, Y+ \beta\in\Gamma(\mathsf P_M)$.

An almost Dirac structure $\mathsf D$ on a manifold $M$ is a Dirac
structure if its set of sections is closed under this bracket, i.e.,
$[ \Gamma(\mathsf{D}), \Gamma(\mathsf{D}) ] \subset \Gamma(\mathsf{D})
$.  Since $\left\langle X+ \alpha, Y+ \beta \right\rangle_M = 0$ if
$(X, \alpha), (Y, \beta) \in \Gamma(\mathsf{D})$, integrability of the
Dirac structure is expressed relative to a non-skew-symmetric bracket
that differs from \eqref{wrong_bracket} by eliminating in the second
line the third term of the second component. This truncated expression
is called the \emph{Dorfman bracket} in the literature:
\begin{equation}\label{Courant_bracket}
\llbracket X+\alpha, Y+\beta \rrbracket  
= [X, Y]+\left(\boldsymbol{\pounds}_{X} \beta - \ip{Y} \dr\alpha \right)
\end{equation} 
for all $X+\alpha, Y+\beta\in\Gamma(\mathsf D)$.
 
\subsection{Generalized complex structures}
\label{sec:gcs}

Let $V$ be a vector space. Consider a linear endomorphism $\mathcal J$
of $V\oplus V^*$ such that $\mathcal J^2=-\Id_V$ and $\mathcal J$ is
orthogonal with respect to the inner product
$$(X+\xi, Y+\eta)_V
=\xi(Y)+\eta(X), \quad \forall X,Y\in V, \; \xi, \eta \in V^* .$$ Such
a linear map is called a \emph{linear generalized complex structure}
by Hitchin \cite{Hitchin03}. The complexified vector
space $(V\oplus V^*)\otimes \mathbb{C}$ decomposes as the direct
sum $$(V\oplus V^* )\otimes \mathbb{C}=E_+\oplus E_-$$ of the
eigenbundles of $\mathcal J$ corresponding to the eigenvalues $\pm
\imaginary$ respectively, i.e.,
$$E_{\pm}=\left\{(X+\xi)\mp \imaginary \mathcal J(X+\xi) \mid X+\xi\in V\oplus V^* \right\}.$$
Both eigenspaces are maximal isotropic with respect to $\langle
\cdot\,,\cdot\rangle$ and they are complex conjugate to each other.

The following lemma is obvious.
\begin{lem}
  The linear generalized complex structures are in 1-1 correspondence
  with the splittings $(V\oplus V^*)\otimes \mathbb{C}=E_+\oplus E_-$
  with $E_{\pm}$ maximal isotropic and $E_-=\cc{E_+}$.
\end{lem}

Now, let $M$ be a manifold and $\J$ a bundle endomorphism of $\mathsf
P_ M=TM\oplus T^*M$ such that $\J^2=-\Id_{\mathsf P_M}$, and $\J$ is
orthogonal with respect to $\langle \cdot\,,\cdot\rangle_M$.  Then
$\J$ is a \emph{generalized almost complex structure}.  In the
associated eigenbundle decomposition $$T_\mathbb{C} M\oplus
T_\mathbb{C}^*M =E_+\oplus E_- ,$$ if $\Gamma(E_+)$ is closed under
the (complexified) Courant bracket, then $E_+$ is a (complex) Dirac
structure on $M$ and one says that $\J$ is a \emph{generalized complex
  structure} \cite{Hitchin03, Gualtieri03}.  In this case, $E_-$ must
also be a Dirac structure since $E_-=\cc{E_+}$. Indeed $(E_+, E_-)$ is
a complex Lie bialgebroid in the sense of Mackenzie-Xu \cite{MaXu94},
in which $E_+$ and $E_-$ are complex conjugate to each other.

\begin{defn}
  Let $\J: \mathsf P_ M\to \mathsf P_ M$ be a vector bundle
  morphism.  Then the \emph{generalized Nijenhuis tensor}
  associated to $\J$ is the map
$$\mathcal N_{\J}:\mathsf P_ M\times_M \mathsf P_ M\to \mathsf P_ M$$
defined by
$$\mathcal N_{\J}(\xi, \eta)
=\llbracket\J\xi, \J\eta\rrbracket+\J^2\llbracket\xi, \eta\rrbracket-\J\left(\llbracket\J\xi,
  \eta\rrbracket+\llbracket\xi,\J\eta\rrbracket\right)$$ for all $\xi,\eta\in\Gamma(\mathsf
P_M)$, where the bracket is the Courant-Dorfman bracket.
\end{defn}
The following proposition gives two equivalent definitions of a
generalized complex structure.

\begin{prop}
  A generalized complex structure is equivalent to any of the
  following:
\begin{enumerate}
\item A bundle endomorphism $\J$ of $\mathsf P_ M$ such that $\J$ is
  orthogonal with respect to $\langle \cdot\,,\cdot\rangle_M$,
  $\J^2=-\Id$ and $\mathcal N_\J=0$.
\item A complex Lie bialgebroid $(E_+, E_-)$ whose double is the
  standard Courant algebroid $\TC M\oplus \TC^*M$, and $E_+$ and $E_-$
  are complex conjugate to each other.
\end{enumerate}
\end{prop}

For a two-form $\omega$ on $M$ we denote by $\omega^\flat: TM\to T^*M$
the bundle map $X\mapsto i_{X}\omega$, while for a bivector $\pi$ on
$M$ we denote by $\pi^{\sharp}: T^*M\to TM$ the contraction with
$\pi$.  Also, we denote by $[\cdot, \cdot]_{\pi}$ the bracket defined
on the space of 1-forms on $M$ by
\begin{equation}
  \label{br-pi} 
  [\xi, \eta]_{\pi}= L_{\pi^{\sharp}\xi}\eta- L_{\pi^{\sharp}\eta}\xi - d\pi(\xi, \eta) 
\end{equation}
for all $\xi,\eta\in\Omega^1(M)$.

A generalized complex structure $\mathcal{J}:\mathsf P_M\to \mathsf
P_M$ can be written
\begin{equation}
\label{J} 
\mathcal{J}= \left( \begin{array}{ll}
    N & \pi^{\sharp} \\
    \omega^\flat & - N^*  
        \end{array}
   \right) 
\end{equation}
where $\pi$ is a bivector field on $M$, $\omega$ is a two-form on $M$,
and $N: TM\to TM$ is a bundle map, satisfying together a list of
identities (see \cite{Crainic04}).  In particular $\pi$ is a Poisson
bivector field.

Let $\mathcal I:\mathsf P_M\to \mathsf P_M$ be the endomorphism
$$\mathcal I=\begin{pmatrix}
  \Id_{TM}&0\\
  0&-\Id_{T^*M}\end{pmatrix}.
$$
Then we have 
\begin{align*}
  \langle \mathcal I(\cdot)\,, \mathcal I(\cdot)\rangle_M
  &=-\langle \cdot\,, \cdot\rangle_M,\\
  \llbracket\mathcal I(\cdot)\,, \mathcal I(\cdot)\rrbracket
  &=\mathcal I\llbracket \cdot\,, \cdot\rrbracket
\end{align*}
and the following proposition follows.

\begin{prop}
  If $\J$ is an almost generalized complex structure on $M$, then
 \begin{equation}\label{def_of_bar_J}
\bar\J:=\I\circ \J\circ \I
\end{equation}
is an almost generalized complex structure. Furthermore,
$$\mathcal N_{\bar\J}=\I\circ \mathcal N_\J\circ(\mathcal I,\mathcal I).$$
Hence, $\bar\J$ is a generalized complex structure if and only if $\J$
is a generalized complex structure.
\end{prop}

The following are two standard examples \cite{Hitchin03}.

\begin{examples}
\begin{enumerate}
\item Let $J$ be an almost complex structure on $M$.
  Then $$\J=\begin{pmatrix} J & 0 \\ 0 & -J^* \end{pmatrix}$$ is
  $\langle\cdot\,,\cdot\rangle_M$-orthogonal and satisfies
  $\J^2=-\Id$.  $\J$ is a generalized complex structure if and only if
  $J$ is integrable.
\item Let $\omega$ be a nondegenerate 2-form on $M$.
  Then $$\J=\begin{pmatrix} 0 & -\left(\omega^{\bemol}\right)^{-1} \\
    \omega^{\bemol} & 0 \end{pmatrix} $$ is a generalized complex
  structure if and only if $d\omega=0$, i.e., $\omega$ is a symplectic
  $2$-form.
\end{enumerate}
\end{examples}

\subsection{Pontryagin bundle over a Lie groupoid}
\paragraph{\textbf{The tangent prolongation of a Lie groupoid}}

Let $\Gamma\rr M$ be a Lie groupoid. Applying the tangent functor to
each of the maps defining $\Gamma$ yields a Lie groupoid structure on
$T\Gamma$ with base $TM$, source $T\s$, target $T\tg$ and
multiplication $T\m:T(\Gamma\times _M\Gamma)\to T\Gamma$.  The
identity at $v_p\in T_pM$ is $1_{v_p}=T_p\epsilon v_p$.  This defines
the \emph{tangent prolongation $T\Gamma\rr TM$ of $\Gamma\rr M$} or
the \emph{tangent groupoid associated to $\Gamma\rr M$}.

\paragraph{\textbf{The cotangent Lie groupoid defined by a Lie groupoid}}
If $\Gamma\rr M$ is a Lie groupoid with Lie algebroid $A\Gamma\to M$,
then there is also an induced Lie groupoid structure on
$T^*\Gamma\rr\, A^*\Gamma= (TM)^\circ$. The source map
$\hat\s:T^*\Gamma\to A^*$ is given by
\[\hat\s(\alpha_g)\in A_{\s(g)}^* \text{ for } \alpha_g\in T_g^*\Gamma,\qquad 
\hat\s(\alpha_g)(a(\s(g)))=\alpha_g(a^l(g))\] for all $a\in
\Gamma(A)$, and the target map $\hat\tg:T^*\Gamma\to A^*$ is given
by \[\hat\tg(\alpha_g)\in A_{\tg(g)}^*, \qquad
\hat\tg(\alpha_g)(a(\tg(g))) =\alpha_g(a^r(g))\] for all $a\in
\Gamma(A)$.  If $\hat\s(\alpha_g)=\hat\tg(\alpha_h)$, then the product
$\alpha_g\star\alpha_h$ is defined by
\[(\alpha_g\star\alpha_h)(v_g\star v_h)=\alpha_g(v_g)+\alpha_h(v_h)
\]
for all composable pairs $(v_g,v_h)\in T_{(g,h)}(\Gamma\times_M
\Gamma)$.

This Lie groupoid structure was introduced in \cite{CoDaWe87} and is
explained for instance in \cite{CoDaWe87,Pradines88,
  Mackenzie05}. Note that the original definition was the following:
let $\Lambda_\Gamma$ be the graph of the partial multiplication
$\mathsf m$ in $\Gamma$, i.e.,
$$\Lambda_\Gamma=\{(g,h,g\star h)\mid g,h\in \Gamma, \s(g)=\tg(h)\}.
$$
The isomorphism $\psi: (T^*\Gamma)^3\to (T^*\Gamma)^3$,
$\psi(\alpha,\beta,\gamma)= (\alpha,\beta,-\gamma)$ sends the conormal
space $(T\Lambda_G)^\circ\subseteq (T^*\Gamma)^3\an{\Lambda_\Gamma}$
to a submanifold $\Lambda_*$ of $(T^*\Gamma)^3$.  It is shown in
\cite{CoDaWe87} that $\Lambda_*$ is the graph of a groupoid
multiplication on $T^*\Gamma$, which is exactly the multiplication
defined above.

\paragraph{\textbf{The ``Pontryagin groupoid'' of a Lie groupoid}}
If $\Gamma\rr M$ is a Lie groupoid with Lie algebroid $A\to M$,
according to \cite{Mehta09}, there is hence an induced VB-Lie groupoid
structure on $\mathsf P_ \Gamma=T\Gamma\oplus T^*\Gamma$ over
$TM\oplus A^*$, namely, the product groupoid, where $T\Gamma\oplus
T^*\Gamma$ and $TM\oplus A^*$ are identified with the fiber products
$T\Gamma \times_\Gamma T^*\Gamma$ and $TM\times_M A^*$, respectively.
It is called a Courant groupoid by Mehta \cite{Mehta09}.

\begin{prop}
Let $\Gamma\rr M$ be a Lie groupoid with Lie algebroid $A\to M$.  Then
the Poytryagin bundle $\mathsf P_ \Gamma=T\Gamma\oplus T^*\Gamma$ is a
Lie groupoid over $TM\oplus A^*$, and the canonical projection
$\mathsf P_ \Gamma\to \Gamma$ is a Lie groupoid morphism.
\end{prop}

We will write $\T\tg$ for the target map
\[\begin{array}{cccc}
\T\tg:&\mathsf P_ \Gamma&\to& TM\oplus A^*\\
&(v_g,\alpha_g)&\mapsto&\left(T\tg(v_g),\hat\tg(\alpha_g)\right)
\end{array},
\]
$\T\s$ for the source map
\[
\T\s:\mathsf P_ \Gamma\to TM\oplus A^*
\]
and $\T\epsilon$, $\T\iota$, $\T\mathsf \m$ for the embedding of the
units, the inversion map and the multiplication of this Lie groupoid,
i.e.
for instance 
\[\T\s(e_g)=(s_*,\hat s)(e_g)\in T_{\s(g)}M\times A^*_{\s(g)}\]
for all $e_g\in\mathsf P_\Gamma(g)$. 

\subsection{Pontryagin bundle over a  Lie algebroid}
 
Given any vector bundle $q_A\colon A\lon M$, the map $Tq_A\colon
TA\lon TM$ has a vector bundle structure obtained by applying the
tangent functor to the operations in $A\lon M$. The operations in
$TA\lon TM$ are consequently vector bundle morphisms with respect to
the tangent bundle structures in $TA\lon A$ and $TM\lon M$ and $TA$
with these two structures is therefore a double vector bundle which we
call the {\em tangent double vector bundle of} $A\lon M$ (see
\cite{Mackenzie92} and references given there).

If $(q_A:A\to M, [\cdot\,,\cdot]_A, \rho_A)$ is a Lie algebroid, then
there is a Lie algebroid structure on $Tq_A\colon TA\lon TM$ defined
in \cite{MaXu94} with respect to which $p_A\colon TA\lon A$ is a Lie
algebroid morphism over $p_M\colon TM\lon M$; we now call this the
{\em tangent prolongation} of $A\lon M$.

For a general vector bundle $q\colon A\lon M$, there is also a double
vector bundle
\begin{equation}\label{diag:T*A}
\begin{xy}
\xymatrix{ T^*A\ar[d]_{c_A}\ar[r]^{r_A}&A^*\ar[d]^{q^*}\\
A\ar[r]_{q_A}&M}
\end{xy}.
\end{equation}
Here the map $r_A$ is the composition of the Legendre transformation
$T^*A\to T^*A^*$ with the projection $T^*A^*\to A^*$ \cite{MaXu94}.

Elements of $T^*A$ can be represented locally as $(\omega,a,\phi)$
where $\omega\in T^*_mM,\ a\in A_m,\ \phi\in A^*_m$ for some $m\in M$.
In these terms the Legendre transformation
\[R:T^*A\to T^*A^*\] can be defined by $R(\omega,\phi,a) =
(-\omega,a,\phi)$; for an intrinsic definition see \cite{MaXu94}. This
$R$ is an isomorphism of double vector bundles preserving the side
bundles; that is to say, it is a vector bundle morphism over both $A$
and $A^*$.  Since $A$ is a Lie algebroid, its dual $A^*$ has a linear
Poisson structure, and the cotangent space $T^*A^*$ has a Lie
algebroid structure over $A^*$.  Hence, there is a unique Lie
algebroid structure on $r_A:T^*A\to A^*$ with respect to which the
projection $T^*A\to A$ is a Lie algebroid morphism over $A^*\to M$.
As a consequence, one can form the product Lie algebroid $TA\times_A
T^*A\to TM\times_M A^*$. Thus we have the following

\begin{prop}
  Let $A\to M$ be a Lie algebroid. Then the Pontryagin bundle $P_A$ is
  naturally a Lie algebroid.  Moreover, the canonical projection
  $P_A\to A$ is a Lie algebroid morphism.
\end{prop}

The double vector bundle $(P_A, TM\oplus A^*, A, M)$ is a VB-Lie
algebroid in the sense of Gracia-Saz and Mehta \cite{GrMe10a}.

\subsection{Canonical identifications}\label{canonical_id_subsec}
The canonical pairing $\lbag\cdot\,,\cdot\rbag_A:A\oplus A^*\to \R$
induces a nondegenerate pairing $\langle\langle \cdot\,,\cdot
\rangle\rangle_A=\pr_2\circ T\lbag\cdot\,,\cdot\rbag_A$ on
$TA\times_{TM}TA^*$, where $\pr_2:T\R=\R\times\R\to\R$ (see
\cite{Mackenzie05}):
\begin{displaymath}
\begin{xy}
  \xymatrix{ A\oplus A^*\ar[d]_{\lbag\cdot\,,\cdot\rbag_A}&&TA
    \times_{TM}TA^*
    \ar[rrd]^{\quad \langle\langle \cdot\,,\cdot \rangle\rangle_A}\ar[d]_{T\lbag\cdot\,,\cdot\rbag_A}&&\\
    \R&&T\R\ar[rr]_{\pr_2}&&\R }
\end{xy}
\end{displaymath}
That is, if $\chi=\left.\frac{d}{dt}\right\an{t=0}\varphi(t)\in TA^*$
and $\xi=\left.\frac{d}{dt}\right\an{t=0}a(t)\in TA$ are such that
$Tq_A(\xi)=Tc_A(\chi)$, then $\langle\langle
\chi,\xi\rangle\rangle_A=\left.\frac{d}{dt}\right\an{t=0}\lbag
a(t),\varphi(t)\rbag_A$.  For instance, if $X\in\Gamma(A)$ and
$\xi\in\Gamma(A^*)$, then $TX\in\Gamma_{TM}(TA)$ and
$T\xi\in\Gamma_{TM}(TA^*)$ are such that
$Tq_A(TX)=\Id_{TM}=Tc_A(T\xi)$ and we have for all $v_p=\dot c(0)\in
T_pM$:
\begin{equation}\label{eq1}
  \langle\langle TX(v_p), T\xi(v_p)\rangle\rangle _A
  =\left.\frac{d}{dt}\right\an{t=0}\lbag X, \xi\rbag_A(c(t))=\pr_2(T_p(\xi(X))(v_p)).
\end{equation}

If $(Tq_A)^\vee:(TA)^\vee\to TM$ is the vector bundle that is dual to
the vector bundle $Tq_A:TA\to TM$, there is an induced isomorphism $I$
\begin{displaymath}
\begin{xy}
\xymatrix{
TA^*\ar[d]_{Tc_A}\ar[r]^{I} &(TA)^\vee\ar[d]^{(Tq_A)^\vee}\\
TM\ar[r]&TM}
\end{xy}
\end{displaymath} that is
defined by 
\begin{equation}\label{def_of_I}
\lbag\! \xi, I(\chi)\rbag_{TA\to TM}= \langle\langle \xi, \chi\rangle\rangle_A
\end{equation}
for all $\chi\in TA^*$ and $\xi\in TA$ such that
$Tc_A(\chi)=Tq_A(\xi)$. That is, the following diagram commutes:

\begin{displaymath}
\begin{xy}
  \xymatrix{ TA\times_{TM}TA^*\ar[rr]^{\qquad
      T\lbag\cdot\,,\cdot\rbag_A}
    \ar[d]_{(\Id,I)}\ar[drr]^{\langle\langle
      \cdot\,,\cdot\rangle\rangle_A}
    &&T\R\ar[d]^{\pr_2}\\
    TA\times_{TM}(TA)^\vee\ar[rr]_{\qquad
      \lbag\cdot\,,\cdot\rbag_{TA}}&&\R }
\end{xy}
\end{displaymath}
\bigskip

Applying this to the case $q_A=p_{TM}$, we get an isomorphism 
\begin{displaymath}
\begin{xy}
\xymatrix{
T(T^*M)\ar[d]_{Tc_M}\ar[r]^{I} &(TTM)^\vee\ar[d]^{(Tp_M)^\vee}\\
TM\ar[r]&TM}
\end{xy}
\end{displaymath}

We have also the canonical involution
\begin{displaymath}
\begin{xy}
\xymatrix{
TTM\ar[d]_{Tp_M}\ar[r]^{\sigma} &TTM\ar[d]^{p_{TM}}\\
TM\ar[r]&TM}
\end{xy}
\end{displaymath} 

Recall that for $V\in\mx(M)$ the map $TV:TM\to TTM$ is a section of
$Tp_M:TTM\to TM$ and $\sigma(TV)$ is a section of $p_{TM}:TTM\to TM$,
i.e. a vector field on $TM$.

We get  an isomorphism $\varsigma:=\sigma^*\circ I: T(T^*M)\to T^*(TM)$ 
\begin{displaymath}
\begin{xy}
\xymatrix{
T(T^*M)\ar[d]_{Tc_M}\ar[r]^{\varsigma} &T^*(TM)\ar[d]^{c_{TM}}\\
TM\ar[r]&TM}
\end{xy}
\end{displaymath}

\begin{prop}
  The map $\Sigma:=(\sigma,\varsigma): T\mathsf P_M\to \mathsf P_{TM}$
\begin{equation}\label{def_of_Sigma}
\begin{xy}
  \xymatrix{
    T\mathsf P_M\ar[d]_{T\pr_M}\ar[r]^{\Sigma} &\mathsf P_{TM}\ar[d]^{\pr_{TM}}\\
    TM\ar[r]&TM}
\end{xy},
\end{equation}
where $\pr_{M}:\mathsf P_M\to M$ is the projection, establishes an isomorphism
of vector bundles.
\end{prop}

\subsection{Lie functor from $\mathsf P_\Gamma$ to $\mathsf P_A$}

\begin{prop}
  Let $\Gamma\rr M$ be a Lie groupoid with Lie algebroid $A$.  Then
  the Lie algebroid of $\mathsf P_\Gamma$ is canonically isomorphic to
  $\mathsf P_A$.

  Moreover, the pairing $\langle\cdot\,,\cdot\rangle_\Gamma$ is a
  groupoid morphism: $\mathsf P_\Gamma\times_\Gamma\mathsf P_\Gamma\to
  \R$.  Its corresponding Lie algebroid morphism coincides with the
  pairing $\langle\cdot\,,\cdot\rangle_A: \mathsf P_A\times_A \mathsf
  P_A \to \R$, under the canonical isomorphism $\mathsf A(\mathsf
  P_\Gamma)\cong \mathsf P_A$.
\end{prop}

This is a standard result. Below we recall its proof which
will be useful in the following.

For any Lie groupoid $\Gamma\rr M$ with Lie algebroid $A\to M$, the
tangent bundle projection $p_\Gamma:T\Gamma\to\Gamma$ is a groupoid
morphism over $p_M:TM\to M$ and applying the Lie functor gives a
canonical morphism $\mathsf A(p_\Gamma):\mathsf A(T\Gamma)\to A$.
This acquires a vector bundle structure by applying $\mathsf A(\cdot)$
to the operations in $T\Gamma\to \Gamma$. This yields a system of
vector bundles
\begin{displaymath}
\begin{xy}
  \xymatrix{
    \mathsf A(T\Gamma)\ar[dd]_{\mathsf A(p_\Gamma)}\ar[rr]^{q_{\mathsf A(T\Gamma)}}&&TM\ar[dd]^{p_M}\\
    &&\\
    A\ar[rr]_{q_A}&&M }
\end{xy}
\end{displaymath}
in which $\mathsf A(T\Gamma)$ has two vector bundle structures, the
maps defining each being morphisms with respect to the other; that is
to say, $\mathsf A(T\Gamma)$ is a double vector bundle.

Associated with the vector bundle $q_A:A\lon M$ is the tangent double
vector bundle
\begin{displaymath} 
\begin{xy}
\xymatrix{
TA\ar[d]_{p_A}\ar[r]^{Tq_A}&TM\ar[d]^{p_M}\\
A\ar[r]_{q_A}&M
}
\end{xy}.    
\end{displaymath}
It is shown in \cite{MaXu94} that the canonical involution
$\sigma:T(T\Gamma)\to T(T\Gamma)$ restricts to a canonical map
$$
\sigma_\Gamma:\mathsf A(T\Gamma)    \to TA
$$
which is an isomorphism of double vector bundles preserving the side
bundles.  Note that here, $A$ is seen as a submanifold of $T\Gamma$.

Similarly, the cotangent groupoid structure $T^*\Gamma\rr A^*$ is
defined by maps which are vector bundle morphisms and, reciprocally,
the operations in the vector bundle $c_\Gamma:T^*\Gamma\lon \Gamma$
are groupoid morphisms. Taking the Lie algebroid of $T^*\Gamma\rr A^*$
we get a double vector bundle
\begin{equation}\label{diag:AT*G}
\begin{xy}
  \xymatrix{
    \mathsf A(T^*\Gamma)\ar[dd]_{\mathsf A(c_\Gamma)}\ar[rr]^{q_{\mathsf A(T^*\Gamma)}}&&A^*\ar[dd]^{p_M}\\
    &&\\
    A\ar[rr]_{q_A}&&M }
\end{xy}
\end{equation}
where the vector bundle operations in $\mathsf A(T^*\Gamma)\lon A$ are
obtained by applying the Lie functor to those in $T^*\Gamma\to
\Gamma$.

It follows from the definitions of the operations in $T^*\Gamma\rr
A^*$ that the canonical pairing
$\lbag\cdot\,,\cdot\rbag_{T\Gamma}:T\Gamma\times_{\Gamma} T^*\Gamma
\to\R$ is a groupoid morphism into the additive group(oid) $\R$.
Hence $\lbag\cdot\,,\cdot\rbag_{T\Gamma}$ induces a Lie algebroid
morphism $\mathsf A(\lbag\cdot\,,\cdot\rbag_{T\Gamma}): \mathsf
A(T\Gamma)\times_A \mathsf A(T^*\Gamma)\to \mathsf A(\R)=\R$.  Note
that $\mathsf A(\lbag\cdot\,,\cdot\rbag_{T\Gamma})$ is the restriction
to $\mathsf A(T\Gamma)\times_A\mathsf A(T^*\Gamma)$ of
$\langle\langle\cdot\,,\cdot\rangle\rangle_{T\Gamma}:
T(T\Gamma)\times_\Gamma T(T^*\Gamma)\to\R$.

As noted in \cite{MaXu94}, $\mathsf
A(\lbag\cdot\,,\cdot\rbag_{T\Gamma})$ is nondegenerate, and therefore
induces an isomorphism of double vector bundles $I_\Gamma:\mathsf
A(T^*\Gamma)\to \mathsf A(T\Gamma)^\vee$, where $\mathsf
A(T\Gamma)^\vee$ is the dual of $\mathsf A(T\Gamma)\to A$. Now
dualizing $\sigma_\Gamma\inv:TA\to \mathsf A(T\Gamma)$ over $A$, we
define
$$
\varsigma_\Gamma = (\sigma_\Gamma\inv)^*\circ I_\Gamma: \mathsf A(T^*\Gamma)\to T^*A;
$$
this is an isomorphism of double vector bundles preserving the side
bundles.  The Lie algebroids $T^*A\to A^*$ and $\mathsf
A(T^*\Gamma)\to A^*$ are isomorphic via $\varsigma_\Gamma$.

The Lie algebroid of the direct sum $\mathsf P_\Gamma=T\Gamma\oplus
T^*\Gamma$ is equal to
\[\mathsf A(\mathsf P_\Gamma)=T_{U}^{\T s}\mathsf P_\Gamma,\]
where we write $U$ for the unit space of $\mathsf P_G$; i.e.,
$U:=TM\oplus A^*$.  By the considerations above, we have a Lie algebroid morphism
$\Sigma_\Gamma=\Sigma\an{\mathsf A(\mathsf
  P_\Gamma)}=\left(\sigma_\Gamma,\varsigma_\Gamma\right)$: 
\begin{displaymath}
\begin{xy}
  \xymatrix{&A\ar[ddd]\ar[rr]^{\Id}&&A\ar[ddd]\\
    \mathsf A(\mathsf P_\Gamma)\ar[ur]\ar[rr]^{\Sigma_\Gamma}
    \ar[ddd]&& \mathsf P_A=TA\oplus T^*A\ar[ddd]\ar[ur]&\\
    \\
    &M\ar[rr]_{\qquad \Id}&&M\\
    TM\oplus A^* \ar[rr]_{\Id}\ar[ur]&&TM\oplus A^*\ar[ur]& }
\end{xy}
\end{displaymath}
preserving the side bundles $A$ and $TM\oplus A^*$.

Recall that we have also a
map \[\Sigma=(\sigma,\varsigma):T\mathsf P_\Gamma\to \mathsf
P_{T\Gamma}.\]
\begin{lem}\label{relation_Sigma_SigmaGamma}
Let $\Gamma\rr M$ be a Lie groupoid and $u$ an element of $\mathsf
A(\mathsf P_\Gamma)\subseteq T\mathsf P_\Gamma$ projection to $a_m\in
A$ and $(v_m,\alpha_m)\in T_mM\times A^*_m$.
Then, if $\Sigma_\Gamma(u)=(v_{a_m},\alpha_{a_m})\in \mathsf P_A(a_m)$ and 
$\Sigma(u)=(\tilde v_{a_m}, \tilde \alpha_{a_m})\in\mathsf
P_{T\Gamma}(a_m)$, we have 
\[\iota_*v_{a_m}=\tilde v_{a_m} \quad \text{ and }\quad \alpha_{a_m}=\tilde\alpha_{a_m}\an{T_{a_m}A}.\]
\end{lem}

\begin{proof}
The first equality follows immediately from the definition of
$\sigma_\Gamma$.

Choose $T_{a_m}A\ni w_{a_m}=\sigma_\Gamma(y)$ for some $y\in\mathsf
A(T\Gamma)\subseteq T(T\Gamma)$ and $\tilde w_{a_m}:=\iota_*w_{a_m}$.
Write also $U=(x,\xi)$ with $x\in\mathsf A(\mathsf P_\Gamma)$ and 
$\xi\in\mathsf A(T^*\Gamma)\subseteq T(T^*\Gamma)$, i.e. $\alpha_{a_m}=\varsigma_\Gamma(\xi)$.
Then we have $\mathsf A(p_\Gamma)(y)=\mathsf A(c_\Gamma)(\xi)$ and we
can compute
\begin{align*}
\lbag  w_{a_m}, \alpha_{a_m}\rbag_{TA}
=&\lbag y, I_\Gamma(\xi)\rbag_{\mathsf A(T\Gamma)}\\
=&\mathsf A(\lbag\cdot\,,\cdot\rbag_{T\Gamma)}(y,\xi)= T(\lbag\cdot\,,\cdot\rbag_{T\Gamma)})(y,\xi)\\
=& \langle\langle y, \xi\rangle\rangle = \lbag y, I(\xi)\rbag_{TP_\Gamma}\\
=& \lbag \sigma(y), \varsigma(\xi)\rbag_{p_{T\Gamma}}=\lbag \tilde
w_{a_m}, \tilde\alpha_{a_m}\rbag_{p_{T\Gamma}}
=\lbag  w_{a_m}, \iota^*\tilde\alpha_{a_m}\rbag_{TA}.
\end{align*}
\end{proof}

The pairing $\langle\cdot\,,\cdot\rangle_\Gamma$ is a groupoid
morphism $\mathsf P_\Gamma\times_\Gamma\mathsf P_\Gamma\to \R$.
Hence, we can consider the Lie algebroid morphism
\[\mathsf A(\langle\cdot\,,\cdot\rangle_\Gamma):\mathsf A(\mathsf P_\Gamma)\times_A\mathsf A(\mathsf P_\Gamma)
\to \mathsf A(\R)=\R.\]

We have 
\begin{equation}\label{A_of_bracket}
  \mathsf A(\langle\cdot\,,\cdot\rangle_\Gamma)=\left(\pr_2\circ T\langle\cdot\,,\cdot\rangle_\Gamma\right)
  \an{\mathsf A(\mathsf P_\Gamma)\times \mathsf A(\mathsf P_\Gamma)}.
\end{equation}

We can see from the proof of the last lemma that
 $\mathsf A(\lbag\cdot\,,\cdot\rbag_{T\Gamma})$ coinsides with
 $\lbag\cdot\,\cdot\rbag_{TA}$
under the isomorphism $\Sigma_\Gamma$.
Hence,
$\mathsf A(\langle\cdot\,,\cdot\rangle_\Gamma)$ coincides with the pairing $\langle\cdot\,,\cdot\rangle_A:
\mathsf P_A\times_A \mathsf P_A \to \R$, under the canonical
isomorphism $\mathsf A(\mathsf P_\Gamma)\cong \mathsf P_A$.

\section{Multiplicative generalized complex geometry}

\subsection{\Glanon groupoids}
\begin{defn}
  Let $\Gamma\rr M$ be a Lie groupoid with Lie algebroid $A\to M$.  A
  multiplicative generalized complex structure on $\Gamma$ is a Lie
  groupoid morphism
\begin{displaymath}
\begin{xy}
  \xymatrix{ T\Gamma\oplus T^*\Gamma\ar[rr]^{\mathcal J}
    \ar@<.6ex>^{\T\s}[dd]\ar@<-.6ex>_{\T\tg}[dd]&& T\Gamma\oplus T^*\Gamma\ar@<.6ex>^{\T\s}[dd]\ar@<-.6ex>_{\T\tg}[dd]\\
    \\
    TM\oplus A^* \ar[rr]_{j_u}&&TM\oplus A^* }
\end{xy}
\end{displaymath}
such that $\J$ is a generalized complex structure.

The pair $(\Gamma\rr M, \mathcal J)$ is then called a \emph{\Glanon
  groupoid}.
\end{defn}

This is equivalent to $\mathsf D_{\mathcal J}$ and $\overline{\mathsf
  D_{\mathcal J}}$, the eigenspaces of $\mathcal J:(T\Gamma\oplus
T^*\Gamma)\otimes\C\to (T\Gamma\oplus T^*\Gamma)\otimes\C$ to the
eigenvalues $i$ and $-i$ being multiplicative Dirac structures on
$\Gamma\rr M$ \cite{Ortiz08t, Jotz12b}.

Note also that $\mathcal J$ is a Lie groupoid morphism if and only if
the maps $N:T\Gamma\to T\Gamma$, $\phi^\sharp:T^*\Gamma\to T\Gamma$,
$N^*:T^*\Gamma\to T^*\Gamma$ and $\omega^\flat:T\Gamma\to T^*\Gamma$
such that
$$\mathcal J=\begin{pmatrix}
  N&\pi^\sharp\\
  \omega^\flat&-N^*\end{pmatrix}
$$
are all Lie groupoid morphisms.

In particular, we have the following

\begin{prop}
  If $\Gamma$ is a \Glanon groupoid, then $\Gamma$ is naturally a
  Poisson groupoid.
\end{prop}

\begin{example}[\Glanon groups]
  Let $G\rr \{*\}$ be a \Glanon Lie group.  Since any multiplicative
  two-form must vanish (see for instance \cite{Jotz12b}), the
  underlying generalized complex structure on $G$ is equivalent to a
  multiplicative holomorphic Poisson structure.  Therefore, \Glanon
  Lie groups are in one-one correspondence with complex Poisson Lie
  groups.
\end{example}

\begin{example}[Symplectic groupoids]
\label{ex:b}
Consider a Lie groupoid $\Gamma\rr M$ equipped with a non-degenerate
two-form $\omega$.  Then the map
 $$\J_\omega=\begin{pmatrix} 0 & -\left(\omega^{\bemol}\right)^{-1} \\ \omega^{\bemol}
   & 0 \end{pmatrix}, $$ where $\omega^{\bemol} : T\Gamma\to
 T^*\Gamma$ is the bundle map $X\mapsto \ip{X}\omega$, defines a
 \Glanon groupoid structure on $\Gamma\rr M$ if and only if $(\Gamma,
 \omega)$ is a symplectic groupoid.
\end{example}

\begin{example}[Holomorphic Lie groupoids]
\label{ex:c}
Let $\Gamma$ be a holomorphic Lie groupoid with $J_\Gamma: T\Gamma \to
T\Gamma$ being its almost complex structure.  Then the map
 $$\J=
 \begin{pmatrix} J_\Gamma & 0 \\0 & -J_\Gamma^* \end{pmatrix}, $$
 defines a \Glanon groupoid structure on $\Gamma$
\end{example}

\subsection{\Glanon Lie algebroids}

Let $A$ be a Lie algebroid over $M$. Recall that $\mathsf P_A
=TA\oplus T^*A$ has the structure of a Lie algebroid over $TM\oplus
A^*$.

\begin{defn}
  A \Glanon Lie algebroid is a Lie algebroid $A$ endowed with a
  generalized complex structure $\mathcal J_A:\mathsf P_ A\to \mathsf
  P_ A$ that is a Lie algebroid morphism.
\end{defn}

\begin{prop}
If $A$ is a \Glanon Lie algebroid, then $(A, A^*)$ is a 
Lie bialgebroid.
\end{prop}

\begin{example}[\Glanon Lie algebra]
Let $\li g$ be a Lie algebra. Then $\mathsf P_{\li g}=\li g\times
( \li g\oplus  \li g^*)$
is a Lie algebroid over $\li g^*$. Hence, a map
\[\mathcal J:\mathsf P_{\li g}\to \mathsf P_{\li g},\]
\[\mathcal J(x,y,\xi)=(x,\J_{x, \li{g}}(y,\xi), \J_{x, \li{g}^*}(y,\xi))\]
can only be a Lie algebroid morphism if $\J_{x, \li{g}^*}(y,\xi)=
\J_{x', \li{g}^*}(y',\xi)$ for all $x,y,x',y'\in\li g$. In particular,
$\J_{x, \li{g}^*}(y,\xi)= \J_{x, \li{g}^*}(y',\xi)$ for all
$x,y,y'\in\li g$ and the map $\mathcal J_x:\{x\}\times\li g\times\li
g^*\to \{x\}\times\li g\times\li g^*$ has the matrix
\[\begin{pmatrix}
n_x&\pi_x^\sharp\\
0&-n_x^*
\end{pmatrix}.\] It thus follows that it must be equivalent to a
complex Lie bialgebra.
\end{example}

\begin{example}[Symplectic Lie algebroid]
\label{ex:bb}
Let $(M, \pi)$ be a Poisson manifold.  Let $A$ be the cotangent Lie
algebroid $A=(T^*M)_\pi$.  Then the map
\[\begin{pmatrix}
0&\left(\omega_A^\flat\right)\inv\\
\omega_A^\flat& 0
\end{pmatrix}\] where $\omega_A$ is the canonical cotangent symplectic
structure on $A$, defines a \Glanon Lie algebroid structure on $A$.
\end{example}

\begin{example}[Holomorphic Lie algebroids]
\label{ex:cc}
Let $A$ be a holomorphic Lie algebroid. Let $A_R$ be its underlying
real Lie algebroid and $j: TA_R\to TA_R$ the corresponding almost
complex structure.  Then the map
\[\begin{pmatrix}
j&0 \\
0& -j^*
\end{pmatrix}\] defines a \Glanon Lie algebroid structure on $A_R$.
\end{example}

\subsection{Integration theorem}

Now we are ready to state the main theorem of this paper.

\begin{thm}
\label{thm:main}
If $\Gamma$ is a \Glanon groupoid with Lie algebroid $A$, then $A$ is
a \Glanon Lie algebroid.

Conversely, given a \Glanon Lie algebroid $A$, if $\Gamma$ is a
$\s$-connected and $\s$-simply connected Lie groupoid integrating $A$,
then $\Gamma$ is a \Glanon groupoid.
\end{thm}

\begin{remark}
  Ortiz shows in his thesis \cite{Ortiz08t} that multiplicative Dirac
  structures on a Lie groupoid $\Gamma\rr M$ are in one-one
  correspondence with morphic Dirac structures on its Lie algebroid,
  i.e. Dirac structures $D_A\subseteq \mathsf P_A$ such that $D_A$ is
  a subalgebroid of $\mathsf P_A\to TM\oplus A^*$ over a set
  $U\subseteq TM\oplus A^*$.

  By extending this result to complex Dirac structures and using the
  fact that a multiplicative generalized complex structure on
  $\Gamma\rr M$ is the same as a pair of tranversal, complex
  conjugated, multiplicative Dirac structures in the complexified
  $\mathsf P_\Gamma$, one finds an alternative method for the proof of
  our main theorem.
\end{remark}

Applying this theorem to Example \ref{ex:b} and Example \ref{ex:bb},
we obtain immediately the following

\begin{thm}
  Let $(P, \pi )$ be a Poisson manifold.  If $\Gamma$ is a
  $\s$-connected and $\s$-simply connected Lie groupoid integrating
  the Lie algebroid $(T^*P)_\pi$, then $\Gamma$ automatically admits a
  symplectic groupoid structure.
\end{thm}

Similarly, applying Theorem \ref{thm:main} to Examples \ref{ex:c} and
\ref{ex:cc}, we obtain immediately the following

\begin{thm}
  If $\Gamma$ is a $\s$-connected and $\s$-simply connected Lie
  groupoid integrating the underlying real Lie algebroid $A_R$ of a
  holomorphic Lie algebroid $A$, then $\Gamma$ is a holomorphic Lie
  groupoid.
\end{thm}

A \Glanon groupoid is automatically a Poisson groupoid, while a
\Glanon Lie algebroid must be a Lie bialgebroid. The following result
reveals their connection

\begin{thm}
  Let $\Gamma$ be a \Glanon groupoid with its \Glanon Lie algebroid
  $A$, $(\Gamma, \pi)$ and $(A, A^*)$ their induced Poisson groupoid
  and Lie bialgebroid respectively. Then the corresponding Lie
  bialgebroid of $(\Gamma, \pi)$ is isomorphic to $(A, A^*)$.
\end{thm}

\subsection{Tangent Courant algebroid}

In \cite{BoZa09}, Boumaiza-Zaalani proved that the tangent bundle of a
Courant algebroid is naturally a Courant algebroid.  In this section,
we study the Courant algebroid structure on $\mathsf P_{TM}$ in terms
of the isomorphism $\Sigma:T\mathsf P_M\to \mathsf P_{TM}$ defined in
\eqref{def_of_Sigma}.

We need to introduce first some notations.

\begin{defn}
\begin{enumerate}
\item The map $\T:\mx(M)\to\mx(TM)$ sends
$V\in\mx(M)$ to $\T V=\sigma(TV)\in\mx(TM)$.
\item The map $\T:\Omega^1(M)\to\Omega^1(TM)$ sends
  $\alpha\in\Omega^1(M)$ to $\T
  \alpha=\varsigma(T\alpha)\in\Omega^1(TM)$.
\item The map $\T:C^\infty(M)\to C^\infty(TM)$
is given by 
$$\T f:=\pr_2\circ T f\in C^\infty(TM,\R),$$
for all functions $f\in C^\infty(M)$.
\end{enumerate}
\end{defn}

Note that if $v_m=\dot c(0)\in T_mM$, 
then 
\begin{equation}\label{def_of_Tf}
  \T f(v_m)=\T f(\dot c(0))=\left.\frac{d}{dt}\right\an{t=0}\left(f\circ c(t)\right),
\end{equation} 
that is, 
\[\T f=\dr f:TM\to \R.\]

Now introduce  the map
$\T:\Gamma(\mathsf P_ M)\to \Gamma(\mathsf P_{TM})$
 given by
$$\T(V,\alpha)=\Sigma(TV, T\alpha)=(\T V, \T\alpha)$$
for all $(V,\alpha)\in\Gamma(\mathsf P_ M)$.

The main result of this section is the following:

\begin{prop}\label{prop_Courant_dorfman}
  For any $e_1,e_2\in\Gamma(\mathsf P_{M})$, we have
\begin{align}
  \left\llbracket\T e_1, \T e_2\right\rrbracket
  &=\T\left\llbracket e_1, e_2\right\rrbracket\label{Courant_bracket_of_lifts}\\
  \langle \T e_1, \T e_2\rangle_{TM} &=\T \left(\langle e_1,
    e_2\rangle_M\right).\label{pairing_of_lifts}
\end{align}
\end{prop}

The following results show that $\T f\in C^\infty(TM)$, $\T
V\in\mx(TM)$ and $\T\alpha\in\Omega^1(TM)$ are the \emph{complete
  lifts} of $f\in C^\infty(M)$, $V\in\mx(M)$ and
$\alpha\in\Omega^1(M)$ in the sense of \cite{YaIs73}.

\begin{lem}
For all $f\in C^\infty(M)$ and $V\in\mx(M)$, we have
\begin{equation}\label{TV_of_Tf}
\T V(\T f)=\T(V(f)).
\end{equation} 
\end{lem}

\begin{proof}
  Let $\phi$ be the flow of $V$. For any $v_m=\dot c(0)\in
  T_mM$, we have
\begin{align*}
  \left(\T V(\T f)\right)(v_m)&= (\T V)(v_m)(\T f)
  =\sigma\left(\left.\frac{d}{dt}\right\an{t=0}\left.\frac{d}{ds}\right\an{s=0}\phi_s(c(t))\right)
  \T f\\
  &=\left(\left.\frac{d}{ds}\right\an{s=0}\left.\frac{d}{dt}\right\an{t=0}\phi_s(c(t))\right)(\T f)\\
  &=\left(\left.\frac{d}{ds}\right\an{s=0}\T f\left(\left.\frac{d}{dt}\right\an{t=0}\phi_s(c(t))\right)\right)\\
  &\overset{\eqref{def_of_Tf}}=\left.\frac{d}{ds}\right\an{s=0}\left.\frac{d}{dt}\right\an{t=0}\left(f\circ\phi_s\right)(c(t))\\
  &=\left.\frac{d}{dt}\right\an{t=0}\left.\frac{d}{ds}\right\an{s=0}\left(f\circ\phi_s\right)(c(t))\\
  &=\left.\frac{d}{dt}\right\an{t=0}V(f)(c(t))\overset{\eqref{def_of_Tf}}=\T(V(f))(v_m).
\end{align*}
\end{proof}

 The following lemma characterizes the sections $\T\alpha$ of $\Omega^1(TM)$.
 \begin{lem}\label{car_of_T_one_forms}
\begin{enumerate}
\item  For all $\alpha\in\Omega^1(M)$ and $V\in\mx(M)$, we have
\begin{align}
  \lbag\T V, \T\alpha\rbag_{TTM}= \T\left(\lbag V,
    \alpha\rbag_{TM}\right).\label{caracterization_of_Theta_T}
\end{align} 
\item $\xi\in\Omega^1(TM)$ satisfies
\begin{align*}
\lbag\T V, \xi\rbag_{T\Gamma}&=0
\end{align*}
for all
$V\in\mx(M)$ if and only if
$\xi=0$.
\end{enumerate}
\end{lem}

\begin{proof}
\begin{enumerate}
\item Choose $V\in\mx(M)$. Then, using $\sigma^2=\Id_{TTM}$:
\begin{align*}
  \lbag\T V, \T\alpha\rbag_{p_{TM}}
 &=\lbag \sigma(T V), (\sigma^*\circ I)(T\alpha)\rbag_{p_{TM}} \\
&= \lbag TV, I(T\alpha)\rbag_{Tp_M}\\
&=\langle\langle TV, T\alpha\rangle\rangle_{TM}\overset{\eqref{eq1}}=
T\lbag V, \alpha\rbag_{TM}.
\end{align*} 
\item Let $\xi\in\Omega^1(TM)$ be such that $\xi(\T V)=0$ for all
  $V\in\mx(M)$.  For any $u\in TM$ with $u\neq 0$ and $v\in T_u(TM)$,
  there exists a vector field $V\in\mx(M)$ such that $\T V(u)=v$. This
  yields $\xi(v)=0$. Therefore, $\xi$ vanishes at all points of $TM$ except
  for the zero section of $TM$. By continuity, we get $\xi=0$.
\end{enumerate}
\end{proof}

Using this, we can show the following formulas.
\begin{lem}\label{tangent_and_cotangent_of_courant}
\begin{enumerate}
\item For all $V,W\in\mx(M)$, we have
$$[\T V, \T W]=\T[V,W].$$
\item For any $\alpha,\beta\in\Omega^1(M)$ and $V,W\in\mx(M)$, we have
\begin{align*}
\ldr{\T V}\T\beta-\ip{\T W}\dr \T\alpha&=
\T\left(\ldr{V}\beta-\ip{W}\alpha\right).
\end{align*}
\end{enumerate}
\end{lem}

\begin{proof}\begin{enumerate}
  \item This is an easy computation, using the fact that if $\phi$ is
    the flow of the vector field $V$, then $T\phi$ is the flow of $\T
    V$ (alternatively, see \cite{Mackenzie05}).
\item For any $U\in\mx(M)$, we can compute 
\begin{align*}
  \lbag \T U, \ldr{\T V}\T\beta-\ip{\T W}\dr \T\alpha\rbag_{p_{TM}} 
=\,&\T V\langle\T\beta, \T U\rangle
  -\langle\T\beta, \T [V,U]\rangle
  -\T W\langle\T\alpha, \T U\rangle\\
  &+\T U\langle\T\alpha, \T W\rangle
  +\langle\T\alpha, \T [W,U]\rangle\\
  =\,&\T V \left(\T\langle\beta, U\rangle\right) -\T\langle\beta,
  [V,U]\rangle
  -\T W\left(\T\langle\alpha, U\rangle\right)\\
  &+\T U\left(\T\langle\alpha, W\rangle\right)
  +\T\langle\alpha, [W,U]\rangle\\
  =\,&\T\left(V \langle\beta, U\rangle\right) -\T\langle\beta,
  [V,U]\rangle
  -\T\left(W\langle\alpha, U\rangle\right)\\
  &+\T\left(U\langle\alpha, W\rangle\right)
  +\T\langle\alpha, [W,U]\rangle\\
  =\,&\T\lbag U, \ldr{V}\beta-\ip{W}\dr \alpha\rbag_{p_M}\\
=\,&\lbag \T U, \T(\ldr{V}\beta-\ip{W}\dr \alpha)\rbag_{p_{TM}}.
\end{align*}
We get 
$$\lbag \T U, \T(\ldr{V}\beta-\ip{W}\dr \alpha)-(\ldr{\T V}\T\beta-\ip{\T W}\dr \T\alpha)\rbag_{p_{TM}}=0$$ for all $U\in\mx(M)$ and we can conclude using Lemma
\ref{car_of_T_one_forms}.\qedhere
\end{enumerate}
\end{proof}

\begin{proof}[Proof of Proposition \ref{prop_Courant_dorfman}]
  Equation \eqref{pairing_of_lifts} follows immediately from
  \eqref{caracterization_of_Theta_T}.

 Formula \eqref{Courant_bracket_of_lifts} for the Dorfman bracket
  on sections of $\mathsf P_{TM}=T(TM)\oplus T^*(TM)$ follows from
  Lemma \ref{tangent_and_cotangent_of_courant}.  (Note that we can
  show in the same manner that the Courant bracket is compatible with
  $\T$.)
\end{proof}

\subsection{Nijenhuis tensor}
Now let $\J:\mathsf P_ M\to \mathsf P_ M$ be a vector bundle
morphism over the identity. Consider the map
$$T\J:T\mathsf P_ M\to T\mathsf P_ M,$$ 
and the map  $\T\J$ defined by
the commutative diagram
\begin{displaymath}\begin{xy}
\xymatrix{ 
T\mathsf P_ M\ar[r]^{T\J}\ar[d]_{\Sigma}&T\mathsf P_ M
\ar[d]^{\Sigma}\\
\mathsf P_{TM}\ar[r]_{\T\J}& \mathsf P_{TM}
}
\end{xy},
\end{displaymath}
i.e., 
$$\T \J=\Sigma\circ T\J\circ\Sigma\inv.$$
Then, by definition, we get for all $e\in\Gamma(\mathsf P_ M)$:
\begin{equation}\label{J_good_with_T}
\T \J\left(\T e\right)
=\left(\Sigma\circ T\J\right)(T e)
=\Sigma(T(\J(e)))=\T(\J(e)).
\end{equation}

The following lemma is immediate.
\begin{lem}\label{T_of_id_and_square}
  \begin{enumerate}
\item $\T(\Id_{\mathsf P_M})=
\Id_{\mathsf P_{TM}}$.
\item Let $\mathcal J:\mathsf P_M\to \mathsf P_M$ be a vector bundle
  morphism over the identity.
Then 
\[\T(\mathcal J^2)=\left(\T(\J)\right)^2.\]
\end{enumerate}
\end{lem}

Given a vector bundle morphism $\J:\mathsf P_ M\to\mathsf P_ M$, we
can consider as above
$$T\mathcal N_{\J}:T\mathsf P_ M\times_{TM} T\mathsf P_ M
\to T\mathsf P_ M.$$
Define
$\T\mathcal N_{\J}:\mathsf P_{TM}\to\mathsf P_{TM}$
by the following commutative diagram:
\begin{displaymath}\begin{xy}
    \xymatrix{ T\mathsf P_ M\times_{TM} T\mathsf P_ M\ar[rr]^{\qquad
        T\mathcal N_{\J}} \ar[d]_{\Sigma\times\Sigma}&&T\mathsf P_ M
      \ar[d]^{\Sigma}\\
      \mathsf P_{TM}\times_{TM}\mathsf P_{TM}\ar[rr]_{\qquad
        \T\mathcal N_{\J}}& &\mathsf P_{TM} }
\end{xy}.
\end{displaymath}
An easy computation using \eqref{J_good_with_T} 
and \eqref{Courant_bracket_of_lifts}
yields 
$$\T\mathcal N_{\J}\left(\T e_1, \T e_2\right)
=\mathcal N_{\T\J}\left(\T e_1, \T e_2\right)$$ for all $e_1,
e_2\in\Gamma(\mathsf P_ M)$.  As in the proof of Lemma
\ref{car_of_T_one_forms}, this implies that $\T\mathcal N_{\J}$ and
$\mathcal N_{\T\J}$ coincide at all points of $TM$ except for the zero
section of $TM$. By continuity, we obtain the following theorem.

\begin{thm}\label{thm_tangent_nijenhuis}
Let $\J:\mathsf P_ M\to \mathsf P_ M$ be a vector bundle morphism. Then 
\begin{equation}\label{eq_formula_for_nijenhuis}
  \T\mathcal N_{\J}=\mathcal N_{\T\J}.
\end{equation}
\end{thm}

\subsection{Multiplicative  Nijenhuis tensor}

Let $\Gamma\rr M$ and $\Gamma'\rr M'$ be Lie groupoids. A map $\Phi:
\Gamma\to \Gamma'$ is a groupoid morphism if and only if the map
$\Phi\times \Phi\times \Phi$ restricts to a map
$$\Lambda_\Gamma\to \Lambda_{\Gamma'},$$
where $\Lambda_\Gamma$  and $\Lambda_{\Gamma'}$ 
are the graphs of the multiplications in $\Gamma\rr M$  and
respectively $\Gamma'\rr M'$.

\medskip

Consider the graphs of the multiplications on 
$T\Gamma$ and $T^*\Gamma$:
\[\Lambda_{T\Gamma}=\{(v_g,v_h,v_g\star v_h)\mid v_g, v_h\in T\Gamma,\, T\tg(v_h)=T\s(v_g)\}=T\Lambda_\Gamma\]
and 
\[\Lambda_{T^*\Gamma}=\{(\alpha_g,\alpha_h,\alpha_g\star \alpha_h)\mid \alpha_g, \alpha_h\in T^*\Gamma,\, \hat\tg(\alpha_h)
=\hat\s(\alpha_g)\}.\] Recall that if
\[\left(\Lambda_{T^*\Gamma}\right)^{\rm op}
=\{(\alpha_g,\alpha_h,-(\alpha_g\star \alpha_h))\mid \alpha_g, \alpha_h\in T^*\Gamma,\, 
\hat\tg(\alpha_h)=\hat\s(\alpha_g)\},\]
i.e.
\[\Lambda_{T\Gamma}\oplus_{\Lambda_\Gamma}\left(\Lambda_{T^*\Gamma}\right)^{\rm op}
=(\Id\times\Id\times\mathcal I)\left(\Lambda_{T\Gamma}\oplus_{\Lambda_\Gamma}\Lambda_{T^*\Gamma}\right),
\]
then
\[\left(\Lambda_{T^*\Gamma}\right)^{\rm op}=(T\Lambda_\Gamma)^\circ.\]

A map $\mathcal J:\mathsf P_\Gamma\to\mathsf P_\Gamma$ is a groupoid morphism
if and only if $\Lambda_{T\Gamma}\oplus_{\Lambda_\Gamma}\Lambda_{T^*\Gamma}$
is stable under the map $\mathcal J\times \mathcal J\times \mathcal J$.
This yields:
\begin{lem}\label{bar_J}
  The map $\mathcal J$ is a groupoid morphism if and only if
  $\Lambda_{T\Gamma}\oplus_{\Lambda_\Gamma}\left(\Lambda_{T^*\Gamma}\right)^{\rm
    op}$ is stable under the map $\mathcal J\times \mathcal J\times
  \bar{\mathcal J}$, where $ \bar{\mathcal J}$ is defined by
  \eqref{def_of_bar_J}.
\end{lem}

Since $\Lambda_{T\Gamma}=T\Lambda_\Gamma$ and
$\left(\Lambda_{T^*\Gamma}\right)^{\rm op}\cong
(T\Lambda_\Gamma)^\circ$, we get that $\mathcal J$ is multiplicative
if and only if
$T\Lambda_{\Gamma}\oplus\left(T\Lambda_{\Gamma}\right)^\circ$ is
stable under $\mathcal J\times \mathcal J\times \bar{\mathcal J}$.

Similarly, the map $\mathcal N_\J:\mathsf P_\Gamma\times_\Gamma\mathsf
P_\Gamma\to\mathsf P_\Gamma$ is multiplicative if and only if
$\mathcal N_\J\times\mathcal N_\J \times\mathcal N_\J$ restricts to a
map
$$(\Lambda_{T\Gamma}\oplus_{\Lambda_\Gamma}\Lambda_{T^*\Gamma})
\times_{\Lambda_\Gamma}(\Lambda_{T\Gamma}\oplus_{\Lambda_\Gamma}\Lambda_{T^*\Gamma})
\to \Lambda_{T\Gamma}\oplus_{\Lambda_\Gamma}\Lambda_{T^*\Gamma}.$$
\begin{lem}\label{bar_N_J}
The map $\mathcal N_\J$ is multiplicative 
if and only if $\mathcal N_\J\times\mathcal N_\J \times\mathcal N_{\bar{\J}}$ restricts to a map
$$\left(T\Lambda_{\Gamma}\oplus_{\Lambda_\Gamma}(T\Lambda_{\Gamma})^\circ
\right)\times_{\Lambda_\Gamma}\left(T\Lambda_{\Gamma}\oplus_{\Lambda_\Gamma}(T\Lambda_{\Gamma})^\circ
\right)\to
T\Lambda_{\Gamma}\oplus_{\Lambda_\Gamma}(T\Lambda_{\Gamma})^\circ.$$ 
\end{lem}

The following lemma is easy to prove.
\begin{lem}\label{lem_submanifold}
  If $M$ is a smooth manifold and $N$ a submanifold of $M$, then the
  Courant-Dorfman bracket on $\mathsf P_ M$ restricts to sections of
  $TN\oplus TN^\circ$.
\end{lem}

Note that $TN\oplus (TN)^\circ$ is a generalized Dirac structure in the
sense of \cite{AlXu01}.  We get the following theorem.
\begin{thm}\label{nijenhuis_also_mult}
Let $(\Gamma\rr M,\mathcal J)$ be an almost generalized complex groupoid.
Then the Nijenhuis tensor $\mathcal N_{\mathcal J}$ is a 
Lie groupoid morphism
$$ \mathcal N_{\mathcal J}:\mathsf P_ \Gamma\times_\Gamma\mathsf P_ \Gamma\to \mathsf P_ \Gamma.$$
\end{thm}

\begin{proof}
  Choose sections $\xi_1, \xi_2, \xi_3, \eta_1, \eta_2,
  \eta_3\in\Gamma(\mathsf P_\Gamma)$ such
  that \[(\xi_1,\xi_2,\xi_3)\an{\Lambda_\Gamma},\,
  (\eta_1,\eta_2,\eta_3)\an{\Lambda_\Gamma}
  \in\Gamma(T\Lambda_{\Gamma}\oplus_{\Lambda_\Gamma}\left(T\Lambda_{\Gamma}\right)^\circ).\]
  Then we have
 \[(\J \xi_1,\J \xi_2,\bar \J\xi_3)\an{\Lambda_\Gamma},\, (\J \eta_1,\J \eta_2,\bar\J\eta_3)\an{\Lambda_\Gamma}
\in\Gamma(T\Lambda_{\Gamma}\oplus_{\Lambda_\Gamma}\left(T\Lambda_{\Gamma}\right)^\circ).\]
From Lemma \ref{lem_submanifold}, it follows that
\[(\mathcal N_\J\times\mathcal N_\J\times \mathcal N_{\bar\J})
\left((\xi_1,\xi_2,\xi_3),  (\eta_1,\eta_2,\eta_3)\right)\]
takes values in 
\[T\Lambda_{\Gamma}\oplus_{\Lambda_\Gamma}\left(T\Lambda_{\Gamma}\right)^\circ\]
on $\Lambda_{\Gamma}$. By Lemma \ref{bar_N_J}, we are done.
\end{proof}

\bigskip

\subsection{Infinitesimal  multiplicative  Nijenhuis tensor}
\begin{defn}
Let $\mathcal J:\mathsf P_\Gamma\to \mathsf P_\Gamma$ be a
Lie groupoid morphism. 
The map 
$$\oplie(\J): \mathsf P_A\to \mathsf P_A
$$
is defined by the commutative diagramm
\begin{displaymath}
\begin{xy}
  \xymatrix{ \mathsf A(\mathsf P_\Gamma)\ar[rr]^{\Sigma_\Gamma}
    \ar[dd]_{\mathsf A(\J)}&&  \mathsf P_A\ar[dd]^{\oplie(\J)}\\
    \\
    \mathsf A(\mathsf P_\Gamma)\ar[rr]_{\Sigma_\Gamma}&& \mathsf P_A }
\end{xy}.
\end{displaymath}
\end{defn}

The following lemma can be found in \cite{BuCadelHo11}.
\begin{lem}\label{lie_is_vb_hom}
Let $\mathcal J:\mathsf P_\Gamma\to \mathsf P_\Gamma$ be a multiplicative 
map. Then
\[\oplie(\J):\mathsf P_A\to \mathsf P_A\]
is a  vector bundle morphism
if and only if  
\[\mathcal J:\mathsf P_\Gamma\to \mathsf P_\Gamma\]
 is a vector bundle morphism.
\end{lem}

Since the map $\mathcal N_\J:\mathsf P_\Gamma\times_\Gamma\mathsf P_\Gamma\to
\mathsf P_\Gamma$
is then also a Lie groupoid morphism 
by Theorem \ref{nijenhuis_also_mult}, we can also consider
\[\oplie(\mathcal N_\J):\mathsf P_A\times_A\mathsf P_A\to \mathsf P_A\]
defined by
\begin{displaymath}
\begin{xy}
\xymatrix{
\mathsf A(\mathsf P_\Gamma)\times_A\mathsf A(\mathsf P_\Gamma)
\ar[rr]^{\quad \Sigma_\Gamma^2}
\ar[dd]_{\mathsf A(\mathcal N_\J)}&&  \mathsf P_A\times_A\mathsf P_A\ar[dd]^{\oplie(\mathcal N_\J)}\\
\\
\mathsf A(\mathsf P_\Gamma)\ar[rr]_{\Sigma_\Gamma}
&& \mathsf P_A
}
\end{xy}.
\end{displaymath}

The main result of this section is the following.

\begin{thm}\label{A_of_nij_and_pairing}
  Let $\mathcal J:\mathsf P_\Gamma\to \mathsf P_\Gamma$ be a vector
  bundle homomorphism and a Lie groupoid morphism.  Then
\begin{enumerate}
\item The map $\oplie(\mathcal J)$ is 
$\langle\cdot\,,\cdot\rangle_{A}$-orthogonal if and only if 
$\J$ is $\langle\cdot\,,\cdot\rangle_{\Gamma}$ -orthogonal
  and,
\item in that case,
\begin{equation}\label{eq_nijenhuis}
\oplie(\mathcal N_{\J})=\mathcal N_{\oplie(\mathcal J)}.
\end{equation}
\end{enumerate}
\end{thm}

For the proof, we need a couple of  lemmas.

\begin{defn}
Let $M$ be a smooth manifold and $\iota:N\hookrightarrow M$ a submanifold of $M$.
\begin{enumerate}
\item A section $e_N=:X_N+\alpha_N$ is $\iota$-related to
  $e_M=X_M+\alpha_M$ if $\iota_*X_N= X_M\an{N}$ and
  $\alpha_N=\iota^*\alpha_M$.  We write then $e_N\sim_{\iota}e_M$.
\item Two  vector bundle morphisms
$\mathcal J_N:\mathsf P_N\to\mathsf P_N$ and  $\mathcal
J_M:\mathsf P_M\to\mathsf P_M$ 
are said to be \emph{$\iota$-related} if 
for each section $e_N$ of $\mathsf P_N$, there 
exists a section $e_M\in\mathsf P_M$ such that
$e_N\sim_{\iota}e_M$ and $\mathcal J_N(e_N)\sim_{\iota}\mathcal J_M(e_M)$.
\item Two vector bundle morphisms 
$\mathcal N_N: \mathsf P_N\times_N\mathsf
  P_N\to \mathsf P_N$ and 
$\mathcal N_M: \mathsf P_M\times_M\mathsf
  P_M\to \mathsf P_M$ 
are \emph{$\iota$-related} if for
  each pair of sections $e_N, f_N\in\Gamma(\mathsf P_N)$,
there exist sections $e_M,f_M\in\Gamma(\mathsf P_M)$ such that
$e_N\sim_{\iota}e_M$, $f_N\sim_{\iota}f_M$ and 
$\mathcal N_N(e_N,f_N)\sim_{\iota}\mathcal N_M(e_M,f_M)$. 
\end{enumerate}
\end{defn}

\begin{lem}\label{StXu08}
Let $M$ be a manifold and $N\subseteq M$ a submanifold.
If $e_N,f_N\in\Gamma(\mathsf P_N)$ are $\iota$-related to
$e_M,f_M\in\Gamma(\mathsf P_M)$, 
then $\llbracket e_N, f_N\rrbracket\sim_{\iota}\llbracket e_M,
f_M\rrbracket$, 
for the Dorfman and the Courant bracket.
\end{lem}

\begin{proof}
This is an easy computation, see also \cite{StXu08}.
\end{proof}

\begin{lem}\label{tensors_related_then_also_nij}
  If $M$ is a manifold, $N\subseteq M$ a submanifold and 
$\mathcal J_N:\mathsf P_N\to\mathsf P_N$
and 
$\mathcal
  J_M:\mathsf P_M\to\mathsf P_M$ two $\iota$-related vector bundle
  morphisms.  Then the generalized
  Nijenhuis tensors $\mathcal N_{\mathcal J_N}$ and $\mathcal
  N_{\mathcal J_M}$ are $\iota$-related.
\end{lem}

\begin{proof}
This follows immediately from Lemma \ref{StXu08} and the definition.
\end{proof}

Recall that the Lie algebroid $A$ of a Lie groupoid $\Gamma\rr M$ is
an embedded 
submanifold of $T\Gamma$, $\iota:A\hookrightarrow T\Gamma$.

\begin{remark}
Note that Lemma \ref{relation_Sigma_SigmaGamma} states 
that if $u\in \Gamma_A(\mathsf A(\mathsf P_\Gamma))$ and $\tilde
u\in\Gamma_{T\Gamma}(T\mathsf P_\Gamma)$ 
is an extension of $u$, then 
\[\Sigma_\Gamma\circ u\sim_\iota \Sigma\circ \tilde u.
\]
\end{remark}

\begin{lem}\label{tensors_on_A_and_TG_related}
  Let $\Gamma\rr M$ be a Lie groupoid and $\mathcal J:\mathsf
  P_\Gamma\to\mathsf P_\Gamma$ a vector bundle morphism.  If $\mathcal J$ is multiplicative, then
\begin{enumerate}
\item $\oplie(\mathcal J):\mathsf P_A\to\mathsf P_A$ and 
 $\T\mathcal
  J:\mathsf P_{T\Gamma}\to\mathsf P_{T\Gamma}$ are $\iota$-related, and
\item  $\oplie (\mathcal N_{\mathcal J})$ and $\T\mathcal N_{\mathcal
  J}$ are $\iota$-related.
\end{enumerate}
\end{lem}

\begin{proof}
\begin{enumerate}
\item Choose a section $e_A$ of $\mathsf P_A$. Then we have
  $\Sigma_\Gamma\inv(e_A)=:u\in\Gamma(\mathsf A(\mathsf P_\Gamma))$
  and since $\mathsf A(\mathsf P_\Gamma)\subseteq T\mathsf P_\Gamma$, 
we find a section $\tilde u$ of $T\mathsf P_\Gamma$ such that $\tilde
u$ restricts to $u$. Set $e_{T\Gamma}:=\Sigma\circ\tilde u$. By Lemma \ref{relation_Sigma_SigmaGamma}, 
we have then $e_A\sim_{\iota}e_{T\Gamma}$. Furthermore, by
construction of $\mathsf A(\mathcal J)$, we know that 
$\mathsf A(\mathcal J)\circ u =(T\mathcal J)\circ \tilde
u$ on  $TM\oplus A^*=TM\oplus TM^\circ\subseteq
\mathsf P_\Gamma$. 
We have then
\[\oplie(\mathcal J)(e_A)=\Sigma_\Gamma\circ \mathsf A(\mathcal
J)\circ u 
\sim_{\iota}\Sigma\circ T\mathcal J\circ \tilde u=\T\mathcal J(e_{T\Gamma}).
\]
\item By definition of $\oplie(\mathcal N_{\mathcal J})$, this can be shown in the same manner.\qedhere
\end{enumerate}
\end{proof}

\begin{proof}[Proof of Theorem \ref{A_of_nij_and_pairing}]
\begin{enumerate}
\item The map $\oplie(\mathcal J)$
is $\langle\cdot\,,\cdot\rangle_{A}$-orthogonal
if and only if 
\[\langle\cdot\,,\cdot\rangle_{A}\circ(\oplie(\J),\oplie(\J))
=\langle\cdot\,,\cdot\rangle_{A}.
\]
We have 
\[\langle\cdot\,,\cdot\rangle_{A}=\oplie(\langle\cdot\,,\cdot\rangle_{\Gamma})
=\mathsf A(\langle\cdot\,,\cdot\rangle_{\Gamma})\circ (\Sigma_\Gamma\inv,\Sigma_\Gamma\inv)\]
by definition and
\begin{align*}
&\langle\cdot\,,\cdot\rangle_{A}\circ(\oplie(\J),\oplie(\J))\\
=\,&\mathsf A(\langle\cdot\,,\cdot\rangle_{\Gamma})\circ (\Sigma_\Gamma\inv,\Sigma_\Gamma\inv)
\circ(\Sigma_\Gamma,\Sigma_\Gamma)\circ(\mathsf A(\J),\mathsf A(\J))\circ (\Sigma_\Gamma\inv,\Sigma_\Gamma\inv)\\
=\,&\mathsf A\bigl(\langle\cdot\,,\cdot\rangle_{\Gamma}\circ(\J,\J)\bigr)\circ (\Sigma_\Gamma\inv,\Sigma_\Gamma\inv).
\end{align*}
Since $\Sigma_\Gamma:\mathsf A(\mathsf P_\Gamma)\to \mathsf P_A$
is an isomorphism, 
we get that
$\oplie(\mathcal J)$
is $\langle\cdot\,,\cdot\rangle_{A}$-orthogonal
if and only if 
\begin{align*}
\mathsf A(\langle\cdot\,,\cdot\rangle_{\Gamma})=\mathsf A\bigl(\langle\cdot\,,\cdot\rangle_{\Gamma}\circ(\J,\J)\bigr)
\end{align*}
and we can conclude.
\item 
 Choose 
sections $e_A, f_A\in\Gamma(\mathsf P_A)$ and 
$u,v\in \Gamma_A(\mathsf A(\mathsf P_{\Gamma}))$ 
such that $e_A=\Sigma_\Gamma\circ u$, $f_A=\Sigma_\Gamma\circ v$.
Choose as in the proof of Lemma \ref{tensors_on_A_and_TG_related}
two extensions $\tilde u$ and $\tilde v\in\Gamma_{T\Gamma}(T\mathsf P_\Gamma)$ 
of $u$ and $v$ and set $e_{T\Gamma}:=\Sigma\circ\tilde u$ and
$f_{T\Gamma}:=\Sigma\circ\tilde v\in\Gamma(\mathsf P_{T\Gamma})$. Then we have 
\[ e_A\sim_{\iota} e_{T \Gamma}\qquad f_A\sim_{\iota}f_{T\Gamma},\]
\begin{equation}\label{equation123}
\oplie(\mathcal J)(e_A)\sim_{\iota} \T\mathcal J(e_{T
  \Gamma})\qquad \oplie(\mathcal J)(f_A)\sim_{\iota}\T\mathcal J(f_{T\Gamma})
\end{equation}
and 
\[\oplie(\mathcal N_{\mathcal J})(e_A, f_A)\sim_{\iota} \T\mathcal
N_{\mathcal J}(e_{T \Gamma}, f_{T\Gamma}).\]
But by Lemma \ref{tensors_related_then_also_nij}, \eqref{equation123} yields also 
\[ \mathcal N_{\oplie(\mathcal J)}(e_A, f_A)\sim_{\iota} \mathcal
N_{\T\mathcal J}(e_{T \Gamma}, f_{T\Gamma}).\]
 Since $\mathcal N_{\T \mathcal J}=\T\mathcal N_{\mathcal J}$
by Theorem \ref{thm_tangent_nijenhuis}, 
we get that 
\[ \oplie(\mathcal N_{\mathcal J})(e_A, f_A)\sim_{\iota}  \mathcal
N_{\T\mathcal J}(e_{T \Gamma}, f_{T\Gamma})
\quad 
\text{ and } \quad 
\mathcal N_{\oplie(\mathcal J)}(e_A, f_A)\sim_{\iota} \mathcal
N_{\T\mathcal J}(e_{T \Gamma}, f_{T\Gamma}).
\]
This yields
\[\oplie(\mathcal N_{\mathcal J})(e_A, f_A)=\mathcal N_{\oplie(\mathcal J)}(e_A, f_A)
\]
and the proof is complete.\qedhere
\end{enumerate}
\end{proof}

The following corollary is immediate.
\begin{cor}\label{lem_eq_nijenhuis}
Let $\mathcal J:\mathsf P_\Gamma\to \mathsf P_\Gamma$ be a vector
bundle morphism and a Lie
groupoid morphism.
Then 
\[\mathcal N_{\oplie(\mathcal J)}=0 \quad \text{ if and only if }\quad 
\mathcal N_{\J}=0.\]
\end{cor}


\subsection{Proof of the integration theorem}
{\it Proof of Theorem \ref{thm:main}}

By Lemma \ref{lie_is_vb_hom}, the map 
$\oplie(\J):\mathsf P_A\to \mathsf P_A$ is a vector bundle
morphism.
Since $\mathcal J^2=-\Id_{\mathsf P_\Gamma}$, we have  using Lemma
\ref{T_of_id_and_square}:
\[\left(\oplie(\mathcal J)\right)^2
=-\Id_{\mathsf P_A}.
\] 
By Corollary \ref{lem_eq_nijenhuis} and Theorem \ref{A_of_nij_and_pairing}, we get
\[\mathcal N_{\oplie(\J)}=0\]
and  
\[\oplie(\J) \text{ is } \langle\cdot\,,\cdot\rangle_A\text{-orthogonal.}\]

Since 
\begin{equation*}
\begin{xy}
\xymatrix{
\mathsf A(\mathsf P_\Gamma)\ar[r]^{\mathsf A(\mathcal J)}\ar[d]& \mathsf A(\mathsf P_\Gamma)\ar[d]\\
TM\oplus A^*\ar[r]&TM\oplus A^*
}
\end{xy}
\end{equation*} is a Lie algebroid morphism
and 
\begin{equation*}
\begin{xy}
\xymatrix{
\mathsf A(\mathsf P_\Gamma)\ar[r]^{\Sigma\an{\mathsf A(\mathsf P_\Gamma)}}\ar[d]& \mathsf P_A\ar[d]\\
TM\oplus A^*\ar[r]&TM\oplus A^*
}
\end{xy}
\end{equation*}
is a Lie algebroid isomorphism over the identity,
the map 
\[\oplie(\J)=\Sigma\an{\mathsf A(\mathsf P_\Gamma)}\circ 
\mathsf A(\mathcal J)\circ \left(\Sigma\an{\mathsf A(\mathsf P_\Gamma)}\right)\inv
=\Sigma\an{\mathsf A(\mathsf P_\Gamma)}\circ 
\mathsf A(\mathcal J)\circ \Sigma\inv\an{\mathsf P_A},\]
\begin{equation*}
\begin{xy}
\xymatrix{
\mathsf P_A\ar[r]^{\oplie(\mathcal J)}\ar[d]& \mathsf P_A\ar[d]\\
TM\oplus A^*\ar[r]&TM\oplus A^*
}
\end{xy}
\end{equation*}
is a Lie algebroid morphism.

For the second part,  consider the map 
\[A_\J:=\Sigma\inv\an{\mathsf P_A}\circ \mathcal J_A\circ \Sigma\an{\mathsf A(\mathsf P_\Gamma)}:
\mathsf A(\mathsf P_\Gamma)\to \mathsf A(\mathsf P_\Gamma).\]
Since $\mathcal J_A:\mathsf P_A\to \mathsf P_A$ is a Lie algebroid
morphism, 
$A_\J$ is a Lie algebroid morphism and there 
is a unique 
Lie groupoid morphism
$\mathcal J:\mathsf P_\Gamma\to \mathsf P_\Gamma$
such that $A_\J=\mathsf A(\J)$. 
By Lemma \ref{lie_is_vb_hom}, $\J$ is a morphism of vector bundles.

We get then immediately $\J_A=\oplie(\J)$.
Since $\J_A^2=-\Id_A$, we get 
$\oplie(\J^2)=-\Id_A=\oplie(-\Id_{\mathsf P_\Gamma})$
and we can conclude by Theorem \ref{A_of_nij_and_pairing}.

This concludes the proof of Theorem \ref{thm:main}.

\section{Application}\label{hol_Poisson_gpd}

\subsection{Holomorphic Lie bialgebroids}
Holomorphic Lie algebroids were studied for various purposes in
the literature. See \cite{Boyom05, EvLuWe99, LaStXu08, Huebschmann99, Weinstein07} 
and references cited there for details. 

By definition, a holomorphic Lie algebroid is a holomorphic vector bundle $A\to X$,
equipped with a holomorphic bundle map $A\xrightarrow{\rho}TX$,
called the anchor map, and a structure of sheaf of complex Lie algebras on $\shs{A}$,
such that
\begin{enumerate}
\item the anchor map $\rho$ induces a homomorphism of sheaves
of complex Lie algebras from $\shs{A}$ to $\Theta_X$;
\item and the Leibniz identity
\[ \lie{V}{fW}=\big(\rho(V)f\big) W+f\lie{V}{W} \]
holds for all $V,W\in\shs{A}(U)$, $f\in\hf{X}(U)$ and
  all open subsets $U$ of $X$.
\end{enumerate}
Here  $\shs{A}$ is  the sheaf of holomorphic sections  of $A\to X$
and $\Theta_X$ denotes   the sheaf of holomorphic vector fields
on $X$.

By forgetting the complex structure,
a holomorphic vector bundle $A \to X $ becomes a real (smooth) vector bundle,
and a holomorphic vector bundle map $\rho: A \to TX $
becomes a real (smooth) vector bundle map.
Assume that $A\to X$ is  a holomorphic vector bundle
 whose underlying real vector
bundle is endowed with a Lie algebroid structure $(A,\rho,\lie{\cdot}{\cdot})$
such that, for any open subset $U\subset X$,
\begin{enumerate}
\item $\lie{\shs{A}(U)}{\shs{A}(U)}\subset\shs{A}(U)$
\item and the restriction of the Lie bracket
$\lie{\cdot}{\cdot}$ to $\shs{A}(U)$ is $\CC$-linear.
\end{enumerate}
Then the restriction of $\lie{\cdot}{\cdot}$ and $\rho$
from $\shs{A}_\infty$ to $\shs{A}$ makes $A$ into
  a holomorphic Lie algebroid, where $\shs{A}_\infty$  denotes
   the sheaf of smooth sections   of $A\to X$.

The following proposition states that any holomorphic Lie algebroid
can be obtained of such a real Lie algebroid in a unique way.

\begin{prop}[\cite{LaStXu08}]
\label{prop:extension}
Given a structure of holomorphic Lie algebroid on the holomorphic
vector bundle $A\to X$ with anchor map $A\xrightarrow{\rho}TX$,
there exists a unique structure of real smooth Lie algebroid on
the vector bundle $A\to X$ with respect to the same anchor map
$\rho$ such that the inclusion of sheaves $\shs{A}\subset\shs{A}_\infty$
is a morphism of sheaves of Lie algebras.
\end{prop}

By $A_R$, we   denote the underlying real Lie algebroid of a holomorphic
 Lie algebroid $A$.  
In the sequel,  by saying that a real Lie algebroid is a holomorphic Lie algebroid,
we mean that it is a holomorphic vector bundle and
its Lie bracket on smooth sections induces a $\CC$-linear
bracket on $\shs{A}(U)$, for all open subset $U\subset X$.

Assume that $(A\to X,\rho,\lie{\cdot}{\cdot})$ is
a holomorphic Lie algebroid. Consider
the bundle map $j: A \to A$ defining
the fiberwise complex structure on $A$. It is simple to see that
 the Nijenhuis torsion of $j$ vanishes \cite{LaStXu08}.
Hence one can define a new (real) Lie algebroid structure on
$A$, denoted by $(A\to X,\rho_{j},\lie{\cdot}{\cdot}_j)$,
where the anchor $\rho_{j}$ is $\rho\rond j$ and
the bracket on $\sections{A}$ is
given by \cite{Kosmann96}
\begin{equation}\label{bracket_on_AI}
 \lie{V}{W}_{j}=\lie{jV}{W}+\lie{V}{jW}-j\lie{V}{W}=-j\lie{jV}{jW}, \qquad\forall V,W\in\sections{A} .
\end{equation}

In the sequel, $(A\to X,\rho_{j},\lie{\cdot}{\cdot}_j)$
will be called the \emph{underlying imaginary Lie algebroid}
and denoted by $A_I$. It is known that
\begin{equation}
j: A_I\to A_R
\end{equation}
 is a Lie algebroid isomorphism \cite{Kosmann96}.

\begin{defn}\label{def:bialg}
A holomorphic Lie bialgebroid is a pair of  holomorphic
Lie algebroids  $(A, A^*)$  over the base $M$ such that  for any open subset $U$
of $M$, the following compatibility holds
\begin{equation}
\label{eq:bialgebroid}
 d_* [X, Y]=[d_* X, Y]+[X,  d_* Y].
\end{equation}
for all $X, Y\in\shs{A}(U)$.
\end{defn}

Given  a holomorphic vector bundle $A\to M$,   
by $\shs{A}^k$ we denote the sheaf of holomorphic sections of
$\wedge^k A\to M$. Then, similar to  the smooth case, $A$ is a holomorphic
Lie algebroid if and only if there is 
a sheaf of Gerstenhaber algebras  $(\shs{A}^\com, \wedge, [\cdot, \cdot])$ over $M$ \cite{LaStXu07}.

\begin{prop}
Let $A\to M$ be a holomorphic vector bundle.
Then $(A, A^*)$ is a holomorphic Lie bialgebroid if and only if 
there is a sheaf of differential Gerstenhaber algebras 
$(\shs{A}^\com, \wedge, [\cdot, \cdot], d_* )$ over $M$.
\end{prop}
\begin{proof}
The proof is exactly the same as in the smooth case. See \cite{Xu99, LaStXu07, Kosmann95}.
For the completeness, we sketch a proof below.

We note that,  if $(A, A^*)$ is a holomorphic Lie bialgebroid, 
 the holomorphic Lie algebroid structure on $A^*$ 
 induces a   complex of sheaves  $d_*: \shs{A}^k\to \shs{A}^{k+1}$ over $M$.
 From the compatibility condition \eqref{eq:bialgebroid} and the Leibniz rule, one can prove
as in \cite{Kosmann95}, that 
 \begin{equation}
\label{eq:bialgebroidf}
d_* [X, f]=[d_* X, f]+[X,  d_* f].
\end{equation}
for all $X\in\shs{A}(U)$ and $f\in \hf{P}(U)$. Therefore by Leibniz rule, one
proves that 
\begin{equation}
\label{eq:bialgebroidff}
 d_* [X, Y]=[d_* X, Y]+[X,  d_* Y]
\end{equation}
for all $X, Y\in \shs{A}^\com (U)$. Thus $(\shs{A}^\com, \wedge, [\cdot, \cdot], d_* )$
is a sheaf of  differential Gerstenhaber algebras  over $M$.
The converse is obvious.
\end{proof}

Let $(A, A^*)$ be  a holomorphic Lie bialgebroid with anchor 
$\rho$ and $\rho_*$, respectively. 
The following proposition shows that the holomorphic bundle map
\begin{equation}\label{def_of_piP}
 \pi_M^\#=(\rho\circ \rho_*^*):T^*M\to
TM.
\end{equation}
is skew-symmetric and defines a  holomorphic Poisson
structure on $M$.

\begin{prop}
Let $(A, A^*)$ be  a holomorphic Lie bialgebroid with anchor
$\rho$ and $\rho_*$, respectively.  Then
\begin{enumerate}
\item $L_{df}X=-[d_*f, X]$ for any $f\in 
\hf{P}(U)$ and $X\in \shs{A}(U)$;
\item $ [d_*f, d_*g]=d_*\{f, g\}$, $\forall f, g\in \hf{P}(U)$;
\item $\rho\smalcirc \rho_*^*=-\rho_*\smalcirc \rho^*$ and
$\pi_M$ defined by  \eqref{def_of_piP} is a holomorphic 
Poisson structure on $M$.
\end{enumerate}
\end{prop}
\begin{proof}
The proof is similar to the proofs of Proposition 3.4, Corollary 3.5 and Proposition 
3.6  in \cite{MaXu94}, and is left to the reader.
\end{proof}

\subsection{Associated real Lie bialgebroids}

By $A_R$ and $A_I$ we denote the underlying real and imaginary
 Lie algebroids of $A$, respectively.
We write  $A_R^*:=(A^*)_R$ and $A_I^*:=(A^*)_I$ for the  underlying real
and imaginary  Lie algebroids of
the complex dual  $A^*$, respectively.

\begin{lem}
Let $A$ be a  holomorphic Lie algebroid.
For any $\alpha\in \gm (A^*)$ and $f\in C^{\infty}(M)$, we have
\begin{align*}
d^I  \alpha&=-(j^*\smalcirc d^R \smalcirc j^*)\alpha, \\
d^I f&=(j^* \smalcirc d^R)f.
\end{align*}
\end{lem}

\begin{proof}
We compute for all $X,Y\in\Gamma(A)$:
\begin{align*}
(d_If)(X)&=\rho_I(X)(f)=\rho(jX)(f)=(d_Rf)(jX),\\
(d_I(j^*\alpha))(X,Y)&=\rho_I(X)\alpha(jY)-\rho_I(Y)\alpha(jX)-\alpha(j[X,Y]_I)\\
&=\rho(j(X))\alpha(jY)-\rho(j(Y))\alpha(jX)-\alpha([j(X),j(Y)])\\
&=(d_R\alpha)(j(X),j(Y)).\qedhere
\end{align*}
\end{proof}

In particular, if the dual $A^*$ of $A$ is also endowed with a Lie
algebroid structure. 
Using the fact that the complex structure in the fibers of $A^*$ is
$j^*$, we have:
\begin{align}
d^I_*  X&=-(j\smalcirc d^R_* \smalcirc j)X, \label{d_*_on_sections}\\
d^I_* f&=(j \smalcirc d^R_*)f\label{d_*_on_functions}
\end{align}
for all $f\in C^\infty(M)$ and $X\in\Gamma(A)$.

\begin{prop}
Let $(A, A^*)$ be a pair of holomorphic Lie algebroids.
The following are equivalent
\begin{enumerate}
\item 
$(A, A^*)$ is a holomorphic Lie bialgebroid;
\item 
$(A_R, A_R^*)$ is a Lie bialgebroid;
\item 
$(A_R, A_I^*)$ is a Lie bialgebroid.
\end{enumerate}
\end{prop}
\begin{proof}
(b)$\Rightarrow$(a)
It is clear that if $(A_R, A_R^*)$ is a real Lie bialgebroid
then the compatibility condition for $(A, A^*)$ being a  
holomorphic Lie bialgebroid is  automatically  satisfied.


(a)$\Rightarrow$(b)
Assume that $U\subset M$ is any open subset, and $X\in \shs{A}(U)$
any holomorphic section. Consider the operator
$C^\infty (U, \C )\to \gm (A_R|_U\otimes \C )$ defined by
\begin{equation}\label{def_of_LX}
\shs{L}_Xf=d_*^R [X, f]-[d_*^R X, f]-[X, d_*^R  f]
\end{equation}
for all $f\in C^\infty (U, \C )$.
Here $d_*^R : \gm (\wedge^\com A_R\otimes \C)\to \gm (\wedge^{\com+1}  A_R\otimes \C)$
is the Lie algebroid cohomology  differential with the trivial coefficients
$\C$ of the  Lie algebroid  $A_R^*$, and
$X$ is considered as a  section of $A_R|_U$.
It is simple to check that $\shs{L}_X$ is a derivation, i.e.,
$$\shs{L}_X (fg)=f\shs{L}_X g+g\shs{L}_X f .$$
Since $(A, A^*)$ is a holomorphic Lie bialgebroid,
it follows that $\shs{L}_X f=0,  \ \forall f\in \hf{P}(U)$
according to \eqref{eq:bialgebroidf}.
Here we use the fact that $d_*f=d_*^R f$ for any $f\in \hf{P}(U)$.
On the other hand,   we also have $\shs{L}_X \cc{f}=0$
since each term in \eqref{def_of_LX} vanishes \cite{LaStXu08}. 

Hence it follows that $\shs{L}_X f=0$ for all $f\in C^\infty (U, \C )$.
Finally,  when $U$ is a contractible open subset
of $M$, $A|_U$ is a trivial bundle,  $\gm (A|_U\otimes \C)$ is spanned by
$\shs{A}(U)$ over $C^\infty (U, \C)$. 
It thus follows that  $d_*[X, Y]=[d_* X, Y]+[X, d_* Y]$
for any $X, Y\in \gm (A|_U\otimes \C)$. Hence 
$(A_R, A_R^*)$ is indeed  a real Lie bialgebroid.

(a)$\Leftrightarrow$(c): The equivalence between (a) and (c) can be proved similarly, using 
the equality $d_*^If=i\cdot d_*f$
for all $f\in\mathcal A(U)$.
\end{proof}

It is well known that Lie bialgebroids are symmetric. Hence
$(A_R, A_R^*)$ is a Lie bialgebroid if and only if
$(A_R^*, A_R)$ is  a Lie bialgebroid.
As a consequence, we have

\begin{cor}
$(A, A^*)$ is a holomorphic Lie bialgebroid iff 
$(A^*, A)$ is a holomorphic Lie bialgebroid.
\end{cor}

\begin{prop} Let $(A, A^*)$ be a pair of holomorphic Lie algebroids.
\begin{enumerate}
\item $(A_R, A_R^*)$ is a Lie bialgebroid, if and only if,
$(A_I, A_I^*)$ is a Lie bialgebroid,
\item $(A_R, A_I^*)$ is a Lie bialgebroid, if and only if,
$(A_I, A_R^*)$ is a Lie bialgebroid.
\end{enumerate}
\end{prop}
\begin{proof}
Recall from \eqref{bracket_on_AI} that the Lie bracket
on $A_I$ is given by 
$\lie{V}{W}_{j}=-j[jV,jW]$ for all $V,W\in\Gamma(A_I)$.
\begin{enumerate}
\item Assume that $(A_R,A_R^*)$ is a Lie bialgebroid. 
Then we have for all $V,W\in\Gamma(A_I)$:
\begin{align*}
d_*^I[V,W]_j&\overset{\eqref{d_*_on_sections}}=(j\smalcirc d_*^R\smalcirc j)(-j[jV,jW])
= (j\smalcirc d_*^R)([jV,jW])\\
&\,=\,j\left(\left[d_*^R(jV),jW\right]+\left[jV,d_*^R(jW)\right]
\right)\\
&\overset{\eqref{d_*_on_sections}}=-j\left(\left[j(d_*^I V),jW\right]+\left[jV,j(d_*^I W)\right]
\right)\\
&\,=\,\left[d_*^I V,W\right]_j+\left[V,d_*^I W\right]_j,
\end{align*}
which shows that $(A_I,A_I^*)$ is a Lie bialgebroid. The converse can
be shown in the same manner.
\item Assume that $(A_I,A_R^*)$ is a Lie  bialgebroid. Then, for all $V,W\in\Gamma(A_R)$, we get
\begin{align*}
d_*^I[V,W]&=-(j\smalcirc d_*^R\smalcirc j)([V, W])= -(j\smalcirc d_*^R\smalcirc j)(-j[j(V), j(W)]_j)\\
&=-(j\smalcirc d_*^R)([j(V), j(W)]_j)\\
&=-j[d_*^R(jV),jW]_j-j[jV,d_*^R(jW)]_j\\
&=-j[j(d_*^IV),jW]_j-j[jV,j(d_*^IW)]_j\\
&=[d_*^IV,W]+[V,d_*^IW].\qedhere
\end{align*}
Hence $(A_R,A_I^*)$ is a Lie bialgebroid. The converse can
be proved similarly.
\end{enumerate}
\end{proof}

\begin{prop}
Let $(A, A^*)$ be a pair of holomorphic Lie algebroids.
Then $(A_R, A_I^*)$ is a Lie bialgebroid if and only if 
$A_R$ endowed with the map $\J_{A_R}:\mathsf P_{A_R}\to \mathsf P_{A_R}$,
\begin{equation}
\J_{A_R}=\begin{pmatrix}
J_{A_R}& \pi_{A_I^*}^\sharp\\
0&-J_{A_R}^*
\end{pmatrix}
\end{equation}
is  a \Glanon Lie algebroid, 
where $\pi_{A_I^*}$ is the Poisson structure on $A_R$ that is induced by the Lie algebroid
structure on $A_I^*$.
\end{prop}

\begin{proof}
Assume first that $(A_R, \J_{A_R})$ is a \Glanon Lie algebroid.
Then the map $\pi_{A_I^*}^\sharp$
is a morphism of Lie algebroids $T^*A\to TA$, and it follows that
$(A_R,A_I^*)$ is a Lie bialgebroid \cite{MaXu00}.

Conversely, if $(A_R,A_I^*)$
is a Lie bialgebroid, the map $\pi_{A_I^*}^\sharp$
is a morphism of Lie algebroids $T^*A\to TA$.
According to Proposition 3.12 in \cite{LaStXu09}
up to a scalar, $\pi_A (\cdot, \cdot)
=\pi_{A_I^*}(\cdot, \cdot)+i\pi_{A_I^*} (J_{A_R}^* \cdot, \cdot)$
is the holomorphic Lie Poisson  structure
on $A$ induced by the holomorphic Lie algebroid $A^*$.
By Theorem 2.7 in \cite{LaStXu08}, $\J_{A_R}$ is hence a generalized complex structure.
Since $A$ is a holomorphic
Lie algebroid, the map $J_{A_R}:TA_R\to TA_R$ is  a morphism of Lie
algebroids according to \cite{LaStXu09}.
\end{proof}

We summarize our results in the following:
\begin{thm}\label{six_eq_things}
Let $(A, A^*)$ be a pair of holomorphic Lie algebroids.
The following assertions are all equivalent:
\begin{enumerate}
\item
$(A, A^*)$ is a holomorphic Lie bialgebroid;
\item
$(A_R, A_R^*)$ is a Lie bialgebroid;
\item
$(A_R, A_I^*)$ is a Lie bialgebroid;
\item
$(A_I, A_R^*)$ is a Lie bialgebroid;
\item
$(A_I, A_I^*)$ is a Lie bialgebroid;
\item the Lie algebroid $A_R$ endowed with the map $\J_{A_R}:\mathsf P_{A_R}\to \mathsf P_{A_R}$,
\begin{equation*}
\J_{A_R}=\begin{pmatrix}
J_{A_R}& \pi_{A_I}^\sharp\\
0&-J_{A_R}^*
\end{pmatrix}
\end{equation*}
is  a \Glanon Lie algebroid.
\end{enumerate}
\end{thm}

\begin{example}
Consider a holomorphic Poisson manifold $(X, \pi)$, where $\pi=
\pi_R+i \pi_I\in \gm (\wedge^2 T^{1,0}X)$ is a holomorphic
Poisson tensor. Let $A=TX$ and $A^*=(T^*X)_\pi$, the
cotangent Lie algebroid associated to $\pi$. Then 
$(A, A^*)$ is a holomorphic Lie bialgebroid.
In this case, $A_R=TX$, $A_I=(TX)_J$, $A_R^*=(TX)^*_{4\pi_R}$
and $A_I^*=(TX)^*_{4\pi_I}$.
\end{example}

\subsection{Holomorphic Poisson groupoids}
\begin{defn}
A holomorphic Poisson groupoid is a holomorphic
Lie groupoid $\gm\rr M$ equipped with 
a holomorphic Poisson structure 
$\pi_\Gamma \in \gm (\wedge^2 T^{1, 0}\gm )$
such that the graph of the multiplication $\Lambda\subset\gm\times\gm\times\bar{\gm}$ 
is a coisotropic   submanifold, where $\bar{\gm}$
stands for  $\gm$ equipped with the opposite Poisson structure.
\end{defn}

Many properties of (smooth) Poisson groupoids have
a straightforward generalization in this holomorphic setting.
 We refer to \cite{Weinstein88b, MaXu94, Xu95}
 for details.
In particular, if $\gm\rr M$ is a holomorphic
Poisson groupoid, then $M$ is naturally a holomorphic
Poisson manifold. More precisely, there exists a unique
holomorphic Poisson structure on $M$ such that the source map
$\s:\Gamma\to M$ is a holomorphic Poisson map, while
the target map is then an anti-Poisson map.

As a consequence, $M$ naturally inherits a holomorphic
Poisson structure.

\begin{thm}[\cite{Xu95,LaStXu07}]
\label{thm:Poissongp1}
Let $\Gamma\rightrightarrows M$
 be a holomorphic Lie groupoid, and
 $\pi_\gm $  a
holomorphic  Poisson tensor on $\gm$. 
Then $(\Gamma,\pi_\gm)$ is a
holomorphic  Poisson groupoid if and only if 
$\pi^\#_\gm\colon T^*\gm \to T\gm$
is a morphism of holomorphic  Lie groupoids over some map
$a_*\colon A^*\to TM$ (which is then the anchor of the Lie algebroid $A^*$).
\end{thm}

 For any  open subset $U\subset M$ and $X
\in\shs{A}^k (U)$, it follows as in 
 \cite[Theorem 3.1]{Xu95}
 that $\left[X^r, \pi_\gm\right]$ is a
right-invariant holomorphic $(k+1)$-vector field on $\gm_U^U$.
Hence it defines an element, denoted $d_* X$, in $\shs{A}^{k+1} (U)$, i.e.
$$(d_* X)^r=\left[X^r, \pi_\gm\right].$$
As  in \cite{Xu99},
 one  proves that
$(\shs{A}^\com, \wedge, [\cdot, \cdot], d_*)$ is a   
 sheaf of differential  Gerstenhaber algebras over $M$. This proves the following proposition.

\begin{prop}
A holomorphic Poisson groupoid naturally induces 
a holomorphic Lie bialgebroid.
\end{prop}

\begin{prop}\label{real_poisson_gpds}
Let $(\gm\rr M, \pi)$ be a holomorphic Poisson groupoid with 
holomorphic  Lie bialgebroid $(A, A^*)$, where
$\pi_\Gamma=\pi_R+i\pi_I \in \gm (\wedge^2 T^{1, 0}\gm )$.
Then
\begin{enumerate}
\item both $(\gm\rr M, \pi_R )$ and 
$(\gm\rr M, \pi_I )$ are Poisson groupoids;
\item the Lie bialgebroids of  $(\gm\rr M, \pi_R )$ and 
$(\gm\rr M, \pi_I )$ are $(A_R, A_{1/4R}^*)$ and
 $(A_R, A_{1/4I}^*)$\footnote{We write $A_{1/4R}^*$ 
(respectively $A_{1/4\cdot I}^*$ for the Lie algebroid
$(A_R^*,1/4[\cdot\,,\cdot]_{A_R^*},1/4\rho_{A_R^*})$
(respectively $(A_I^*,1/4[\cdot\,,\cdot]_{A_I^*},1/4\rho_{A_I^*})$).
}, respectively.
\end{enumerate}
\end{prop} 

\begin{proof}
\begin{enumerate}
\item If $\pi$ is a multiplicative holomorphic Poisson structure 
on $\Gamma\rr M$, then its real and imaginary parts are also multiplicative.
It is shown in \cite{LaStXu08} that both $\pi_R$ and $\pi_I$ are Poisson bivector fields.
\item We have 
$\IM ((T^*\Gamma)_\pi)=(T^*\Gamma)_{4\pi_I}$ 
and 
$\RE((T^*\Gamma)_\pi)=(T^*\Gamma)_{4\pi_R}$ \cite{LaStXu08}. 
The Lie algebroid structure on $(T^*\Gamma)_\pi$
restricts to the holomorphic Lie algebroid structure 
on $A^*$ as follows;
the map $\rho:A^*\to TM$
is just the restriction of $\pi^\sharp$ to $A^*=TM^\circ$
seen as a subbundle of $T_M^*\Gamma$, and the bracket
on $T^*\Gamma$ restricts to a bracket on $A^*$.
In the same manner, the Lie groupoid 
$(\Gamma\rr M,\pi_R^\sharp)$
induces a Lie algebroid structure on ${A_R}^*$
that is the restriction of the Lie algebroid structure on $(T^*\Gamma)_{\pi_R}$.
Hence, we can conclude easily.
\end{enumerate}
\end{proof}

\begin{prop}
\label{prop:4.16}
Let $(\Gamma\rr M, \pi)$ be a holomorphic Poisson groupoid, with the
almost complex structure $J_\Gamma : T\Gamma \to T\Gamma$. Then 
\begin{enumerate}
\item $\mathcal J_\pi:\mathsf P_\Gamma\to\mathsf P_\Gamma$ given by the matrix
\begin{align*}
\begin{pmatrix}
J_\Gamma& \pi_I^\sharp\\
0& -J_\Gamma^*
\end{pmatrix}
\end{align*} 
defines a \Glanon groupoid structure on $\Gamma$.

\item The Lie algebroid morphism $\oplie(\J_\pi): \mathsf P_A\to\mathsf P_A $ 
is given by 
\begin{align*}
\begin{pmatrix}
J_{A_R}& \pi_{A_{1/4\cdot I}^*}^\sharp\\
0& -J_{A_R}^*
\end{pmatrix},
\end{align*} 
where $\pi_{A_{1/4\cdot I}^*}$ is the linear Poisson structure
defined on $A_R$ by the Lie algebroid $A_{1/4\cdot I}^*$. 
\end{enumerate}
\end{prop}
\begin{proof}
\begin{enumerate}
\item By Theorem 2.7 in \cite{LaStXu07}, $\J_\pi$ is a generalized complex structure.
Since  $(\Gamma\rr M, J_\Gamma)$ is a holomorphic Lie groupoid and 
$(\Gamma\rr M,\pi_I)$ is a Poisson groupoid, the maps
$J_\Gamma:T\Gamma\to T\Gamma$, its dual
$J_\Gamma^*:T^*\Gamma\to T^*\Gamma$ 
and $\pi_I^\sharp:T^*\Gamma\to T\Gamma$ are all multiplicative.
\item It is shown in \cite{LaStXu09} that 
$\sigma_\Gamma\circ A(J_\Gamma)\circ \sigma_\Gamma\inv=J_{A_R}$, and 
in \cite{MaXu00}
that 
$\varsigma_\Gamma\circ A(\pi_I^\sharp)\circ\varsigma_\Gamma\inv=\pi_{A_{1/4\cdot I}^*}^\sharp$,
since $(A_R, A_{1/4\cdot I}^*)$ is the Lie bialgebroid of $(\Gamma\rr M,\pi_I)$.
\end{enumerate}
\end{proof}

A holomorphic Lie bialgebroid $(A, A^*)$ is
said to be {\em integrable} if it
is isomorphic to the induced
holomorphic Lie bialgebroid of a holomorphic Poisson groupoid.

As a consequence of Theorem \ref{thm:main}
and Proposition \ref{prop:4.16}, we obtain the
following main result of this section.

\begin{thm}
\label{main}
Given a holomorphic Lie bialgebroid $(A, A^*)$, if the
underlying real Lie algebroid $A_R$ integrates to  a
$\s$-connected and $\s$-simply connected Lie groupoid $\Gamma$,
then $\Gamma$ is a holomorphic Poisson groupoid.
\end{thm}

\begin{remark}
This is proved in \cite{LaStXu09} in the special case where $(A,
A^*)$ is the holomorphic Lie bialgebroid $((T^*M)_\pi, TM)$ induced
from a holomorphic Poisson manifold $(M,\pi)$.
\end{remark}

\bibliographystyle{amsplain}

\def\cprime{$'$} \def\polhk#1{\setbox0=\hbox{#1}{\ooalign{\hidewidth
  \lower1.5ex\hbox{`}\hidewidth\crcr\unhbox0}}}
\providecommand{\bysame}{\leavevmode\hbox to3em{\hrulefill}\thinspace}
\providecommand{\MR}{\relax\ifhmode\unskip\space\fi MR }
\providecommand{\MRhref}[2]{%
  \href{http://www.ams.org/mathscinet-getitem?mr=#1}{#2}
}
\providecommand{\href}[2]{#2}

\end{document}